\documentclass[opre,nonblindrev]{informs3} 

\renewcommand{\theARTICLETOP}{}
\renewcommand{\theARTICLETOPLEFT}{}
\renewcommand{\theARTICLETOPRIGHT}{}

\def\setoddRH{\hbox to \textwidth{\fs.7.8.\tabcolsep0pt
  \begin{tabular*}{\textwidth}[b]{l@{\extracolsep\fill}r}
  {\theRRHFirstLine}&\\ {}&  \raisebox{0pt}[0pt][0pt]{\fs.10.10.\thepage}\\[-4pt]
  \rlap{\VRHDW{0.5pt}{0pt}{\textwidth}}&\\
  \end{tabular*}}}
\def\setevenRH{\hbox to \textwidth{\fs.7.8.\tabcolsep0pt
  \begin{tabular*}{\textwidth}[b]{l@{\extracolsep\fill}r}
  &{\theLRHFirstLine}\\ \raisebox{0pt}[0pt][0pt]{\fs.10.10.\thepage}&{}\\[-4pt]
  \rlap{\VRHDW{0.5pt}{0pt}{\textwidth}}&\\
  \end{tabular*}}}

\DoubleSpacedXI 

\usepackage{endnotes}
\let\footnote=\endnote
\let\enotesize=\normalsize
\def\notesname{Endnotes}%
\def\makeenmark{$^{\theenmark}$}
\def\enoteformat{\rightskip0pt\leftskip0pt\parindent=1.75em
  \leavevmode\llap{\theenmark.\enskip}}

\usepackage{natbib}
 \bibpunct[, ]{(}{)}{,}{a}{}{,}%
 \def\bibfont{\small}%
\TheoremsNumberedThrough     
\ECRepeatTheorems

\EquationsNumberedThrough    

\usepackage{graphicx}
\usepackage{subfigure}

\usepackage{algorithm}
\usepackage{algorithmic}

\usepackage{amsmath}
\usepackage{amsfonts}
\usepackage{amssymb}
\usepackage{latexsym}
\usepackage{psfrag}
\usepackage{float}
\usepackage{color}

\usepackage{bbm}

\usepackage{setspace}
\usepackage{enumerate}
\usepackage{verbatim}
\usepackage{multirow}

\newcommand{\djw}[1]{{\color{black} #1}}
\newcommand{\lp}[1]{{\color{black} #1}}
\newcommand{\rj}[1]{{\color{black} #1}}
\newcommand{\todo}[1]{{}}
\newcommand{\ignore}[1]{}

\newcommand{\reals}{\ensuremath{\mathbb{R}}}

\newcommand{\E}{\ensuremath{\mathbb{E}}}
\renewcommand{\P}{\ensuremath{\mathbb{P}}}

\renewcommand{\epsilon}{\varepsilon}
\renewcommand{\v}[1]{\ensuremath{\boldsymbol{\mathrm{#1}}}}

\renewcommand{\mc}[1]{\ensuremath{\mathcal{#1}}}

\newcommand{\cF}{\ensuremath{\mathcal{F}}}

\newcommand{\FWER}{\ensuremath{\mathsf{FWER}}}

\newcommand{\FDR}{\ensuremath{\mathsf{FDR}}}

\newcommand{\CI}{\ensuremath{\mathsf{CI}}}

\newcommand{\1}{\ensuremath{\mathbbm{1}}}

\begin{document}


\RUNAUTHOR{Johari, Pekelis, and Walsh}

\RUNTITLE{Always Valid Inference}

\TITLE{Always Valid Inference: Continuous Monitoring of A/B Tests}

\ARTICLEAUTHORS{%
\AUTHOR{Ramesh Johari\thanks{RJ was a technical advisor to Optimizely, Inc., when this work was carried out.}}
\AFF{Department of Management Science and Engineering, Stanford University, \EMAIL{rjohari@stanford.edu}} 
\AUTHOR{Leo Pekelis\thanks{LP was employed by Optimizely, Inc., when this work was carried out.}}
\AFF{OpenDoor, Inc, \EMAIL{lpekelis@gmail.com}
}
\AUTHOR{David Walsh\thanks{DJW was employed by Optimizely, Inc., when this work was carried out.}}
\AFF{Department of Statistics, Stanford University, \EMAIL{dwalsh@stanford.edu}}
} 

\ABSTRACT{%
A/B tests are typically analyzed via frequentist p-values and confidence intervals; but these inferences are wholly unreliable if users endogenously choose samples sizes by {\em continuously monitoring} their tests.  We define {\em always valid} p-values and confidence intervals that let users try to take advantage of data as fast as it becomes available, providing valid statistical inference whenever they make their decision.  Always valid inference can be interpreted as a natural interface for a sequential hypothesis test, which empowers users to implement a modified test tailored to them.  In particular, we show in an appropriate sense that the measures we develop tradeoff sample size and power \rj{efficiently}, despite a lack of prior knowledge of the user's relative preference between these two goals.  We also use always valid p-values to obtain multiple hypothesis testing control in the sequential context.  Our methodology has been implemented in a large scale commercial A/B testing platform to analyze hundreds of thousands of experiments to date.
}%



\maketitle

%


\section{Introduction}
\label{sec:introduction}

\begin{figure*}
\begin{center}
\includegraphics[width=0.8\textwidth]{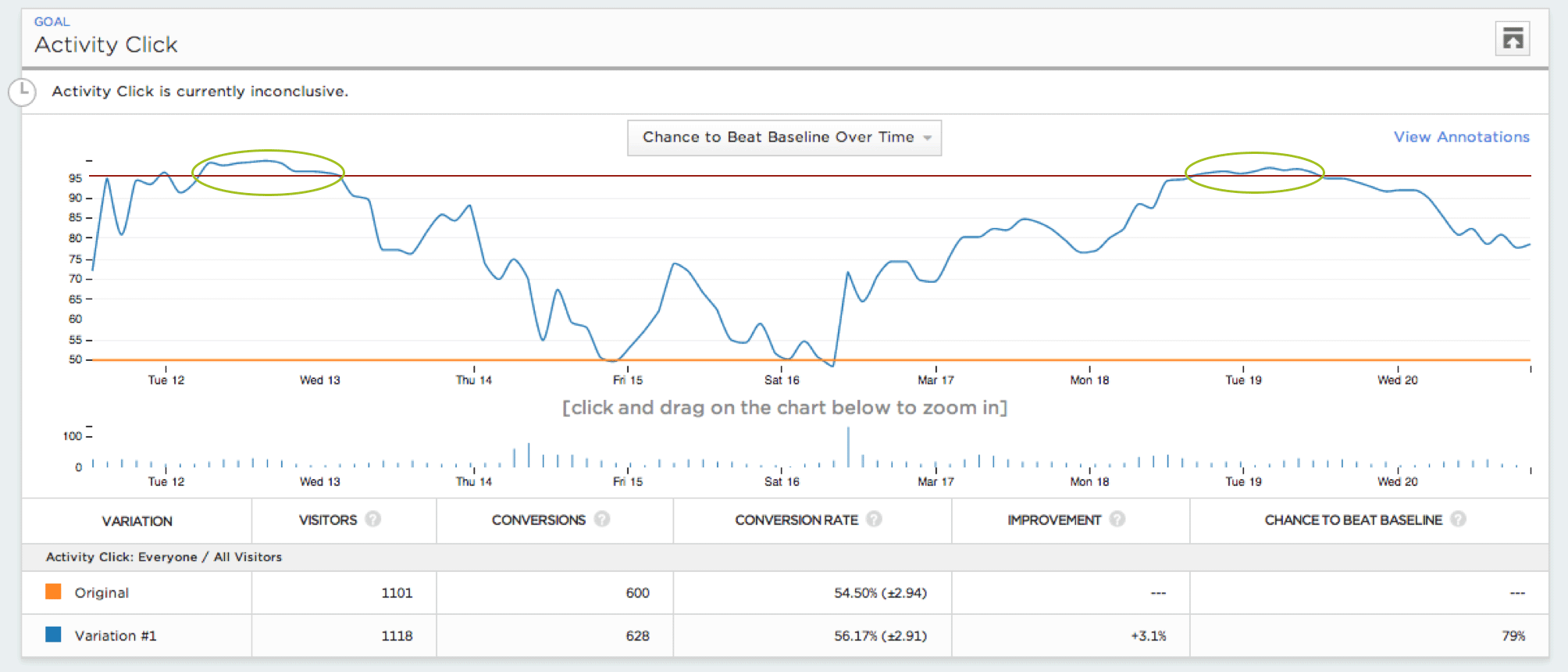}
\end{center}
\caption{A typical dashboard from a large commercial A/B testing platform service.  The graph depicts the ``chance to beat baseline'' of a test, which measures $1 - p_n$ over time, where $p_n$ is the p-value after $n$ observations of the null hypothesis that the clickthrough rate in treatment and control is identical.
This particular test is a {\em A/A test}: both the treatment and control are the {\em same}.  The graph shows that $1-p_n$ rises above the 95\% significance threshold if the user continuously monitors the test, triggering a Type I error.}
\label{fig:aatest}
\end{figure*}

This paper reports on novel statistical methodology underlying the implementation of a large-scale commercial experimentation platform for web applications and services.\footnote{The platform is named {\em Optimizely} (\url{http://www.optimizely.com}).  The methodology of this paper forms the core of the statistical backend of the Optimizely platform.  Throughout the paper, we refer to the platform as ``a large commercial A/B testing platform'' or similar.}  Web applications typically optimize their product offerings using randomized controlled trials (RCTs); in industry parlance this is known as {\em A/B testing}.  The rapid rise of A/B testing has led to the emergence of a number of widely used platforms that handle the implementation of these experiments \citep{kohavi2013online,tang2010overlapping}.

A standard A/B test with two variations ({\em control} and {\em treatment}) involves testing the {\em null hypothesis} that both groups share the same parameter (e.g., customer conversion rate from clicking on a link to a sale) against the {\em alternative hypothesis} that they are different.  The A/B testing platform communicates experimental results to the user via frequentist parameter testing measures, i.e., p-values and confidence intervals. Besides well known properties of their optimality \citep[see, e.g.,][]{lehmann1986testing}, these measures imply an exceptionally simple ``user interface''; the user implements a mechanistic rule (threshold the p-value at their desired Type I error rate $\alpha$) that does not require further knowledge of experiment details. This has two valuable consequences, underpinning their ubiquity in industrial practice: first, the same interface can be employed by {\em many} users, each with their own $\alpha$; and second, experimental results can be analyzed by users without advanced statistical training.



These desirable outcomes only obtain, however, when p-values and confidence intervals are used as intended.  In A/B testing practice, users are not constrained to simply analyze the output of an experiment; they can also adjust the experimental design in response to the data observed.  This type of behavior can entirely undermine statistical reliability. A particularly pervasive form of this behavior is commonplace: users {\em continuously monitor} the p-values and confidence intervals reported in order to set the sample size of an experiment dynamically \citep{miller_blogpost}.  Figure \ref{fig:aatest} shows how a typical dashboard enables such behavior.

The incentive to continuously monitor experiments is strong because the opportunity cost of longer experiments is large. There is value to detecting true effects as quickly as possible, or giving up if it appears that no effect will be detected soon so that they may test something else. Most users further lack good prior understanding of their cost elasticity as well as the effect size they seek, frustrating attempts to optimize run-time in advance. Indeed, the ability to trade off maximum detection with minimum run-time dynamically is a crucial benefit of the availability of real-time data in modern A/B testing environments.


No correction for continuous monitoring is typically made in industrial practice. Consequently, the resulting feedback loop between the statistical output and the experimental design  undermines inferential validity, as computations are performed under the naive assumption of their independence.  Very high false positive probabilities are obtained---well in excess of the nominal $\alpha$ that the user can tolerate. Even with 10,000 samples (a sample size which is quite common in online A/B testing), Type I error can easily increase fivefold.


Our challenge is the following: can we deliver \rj{efficient} inference in a simple interface, but in an environment where users continuously monitor experiments, and where their priorities regarding run-time and detection are not known in advance? {\em We develop a statistical approach that addresses this challenge, and report on an implementation in a commercial A/B testing platform.} A key to our framework is that we employ the {\em same interface} as a traditional A/B testing platform: we present the user with p-values and confidence intervals. However, our measures have the following properties.  {\em First}, Type I error is controlled under any data-dependent rule the user might choose to stop the experiment (i.e., any stopping time for the data). Continuous monitoring does not inflate Type I error.  {\em Second}, in a sense we make precise below, we show that if the user stops when our p-value drops below her personal $\alpha$, the resulting rule approximately obtains \rj{an efficient} trade-off for run-time and detection, even with no advance knowledge of her preferences.

In more detail, our contributions are as follows.  

\begin{enumerate}
\item {\em Type I error control: Always valid p-values and confidence intervals}.  Our first contribution is to develop p-values that control Type I error in a strong sense.  In particular, we ask: what p-value processes control Type I error at any data-dependent stopping time the user might choose?  We refer to such p-values as {\em always valid} p-values (and analogously, always valid confidence intervals).  We show how these always valid p-value processes can be constructed using sequential hypothesis tests \citep{wald1945sequential, siegmund1985sequential, lai2001sequential, siegmund1978estimation}; in particular, we identify a duality between always valid p-values and those sequential tests that do not accept $H_0$ in finite time, known as {\em sequential tests of power one} \citep{RS74}.  Under this duality, the natural policy of stopping the first time that the always p-value process crosses the level $\alpha$ implements the corresponding sequential test of power one.  In this way we retain the simple ``user interface'' of p-values, but guarantee Type I error control under continuous monitoring.

\item {\em Efficiently trading off power and run-time: The mixture sequential probability ratio test}.  Having controlled Type I error, we then ask: how can we efficiently trade off power and run-time?  The challenge for an A/B testing platform, as noted above, is that users' objective functions---and in particular, their relative prioritization of run-time and detection---are not typically known in advance.  \rj{We aim to find always valid p-value processes that lead to an efficient tradeoff for the user.}

It is evident that without restricting the class of user models we consider, no meaningful result is possible; the space of potential user preferences over run-time and detection is vast (and some are unreasonable, e.g., preferring long experiments that do not detect effects).  Instead, we focus on users that generally want short tests and good detection, modeled as follows: the user stops at either the first time the p-value process crosses $\alpha$, or at a fixed maximum {\em failure time} $M$, whichever comes first.  A larger $M$ indicates greater patience, and a corresponding preference for detection of smaller effects.  While this class does not capture all possible objective functions, it does allow us to capture what we consider to be the most interesting axis of user heterogeneity: how much they care about power versus how much they care about run-time.

Our main contribution is to demonstrate that always valid p-values derived from a particular class of sequential tests known as {\em mixture sequential probability ratio tests} (mSPRTs) \citep{robbins1970statistical} achieve an efficient tradeoff between power and run-time for such users to first order in $M$ as $M \to \infty$.  We achieve this result for settings where the data is generated from a single-parameter exponential family.  \rj{This result provides evidence that the mSPRT can produce always valid p-values that yield valuable inference for most users, albeit in the limit where $M$ is very large.  We therefore complement our theory with empirical analysis that compares the mSPRT to other sequential tests at finite values of $M$.  These empirical results demonstrate that the mSPRT delivers high performance in the regime relevant for practical application.}


\rj{The mSPRT involves stopping the first time a mixture of the likelihood ratio of alternative to null crosses a threshold, where the mixture is over potential values of the alternative.  Although first order efficiency as $M \to \infty$ holds for any mixing distribution, the mixing distribution plays an important role in second order performance.  In particular, as power approaches one when $M \to \infty$ for mSPRTs, we choose among this class by optimizing the mixing distribution to minimize run-time.  Formally, in a Bayesian setting (where effect sizes are drawn from a prior), we find a particular choice of mSPRT that minimizes expected run-time in the limit where $M \to \infty$.  The run-time minimizing choice of mSPRT has an appealingly simple form: e.g., when there is a Gaussian prior on effect sizes, the optimal mSPRT approximately matches the variance of the mixing distribution to the variance of the prior distribution of effect sizes.  We also complement this theoretical investigation with numerics to study the sensitivity of performance to the choice of mixing distribution at finite values of $M$.  We find that although the mixing distribution can have a significant impact on performance, there is also robustness, in that a well-chosen mixing distribution can deliver good performance in a wide-range of conditions.}

\item {\em Implementation in a commercial A/B testing platform}.  As noted above, a key contribution of our work is the deployment of our methods in a commercial A/B testing platform, used by thousands of customers worldwide.  The main technical challenge in implementation is that our results above on optimality of the mSPRT are derived in a {\em single-stream} setting, where a single unknown parameter of a data-generating process is being tested; by contrast, in A/B testing, since there are two variations (control and treatment), we are in a {\em two-stream} setting.  We extend our work to this setting, essentially by treating additional parameters besides the true treatment effect as nuisance parameters; this is the method that is deployed in the platform.  We also report both empirically and theoretically on a comparison of our methodology with classical fixed horizon testing.  We find that our approach can deliver equivalent power as fixed horizon testing, with a sublinear sample size, suggesting that even statistically savvy users would receive results faster using our approach in practice.

We \djw{note} further that in a companion conference paper, we provide greater detail on the implementation of our work in the large commercial A/B testing platform described above; the reader is referred to \cite{johari2017kdd}.

\item {\em Multiple hypothesis testing}.  Finally, we note that a major practical advantage of always valid p-values is that we can employ them within existing methodology that uses these measures as inputs to provide error guarantees in multiple hypothesis testing, despite the sequential nature of A/B testing.    In particular, we can control the {\em family-wise error rate} (FWER), as well as the {\em false discovery rate} (FDR) in the sequential setting under some assumptions on the user's stopping time.  We also study {\em false coverage rate} (FCR) control for confidence intervals.  This combination of sequential testing and multiple hypothesis testing corrections is novel to our work, and enabled because of the introduction of always valid p-values.  The resulting FDR- and FCR-controlling procedures have also been implemented in the commercial service described above, as part of the same deployment.
\end{enumerate}


Taken together, we preserve the benefits of p-values and confidence intervals, while modernizing their computation to account for continuous monitoring as well as multiple hypothesis testing.  \rj{The remainder of the paper is organized as follows.  In Section \ref{sec:related}, we present related literature.  In Section \ref{sec:preliminaries}, we describe our basic model and notation.  In Section \ref{sec:always_valid}, we introduce always valid p-values and confidence intervals.  In Section \ref{sec:optimality}, we study the design of efficient always valid p-value processes via the use of the mSPRT, both theoretically and empirically.  In Section \ref{sec:abtests}, we discuss details of deployment and implementation, and particularly the adaptation of our basic theory to two streams of data.  In Section \ref{sec:multiple}, we discuss multiple hypothesis testing.  Finally, we conclude in Section \ref{sec:conclusion}.}

\section{Related work}
\label{sec:related}

\djw{\subsection{Sequential hypothesis testing}}

The desire to test hypotheses using data that arrives sequentially over time is far from new.  Rather, sequential analysis is a mature field with roots dating back to \cite{wald1945sequential}, and sequential tests are widely used in areas such as pharmaceutical clinical trials. For its history, methodology and theory, we direct the reader to the encyclopedic resource of \cite{ghosh1991handbook}; \rj{see also \cite{siegmund1985sequential} for an introduction to the topic.}

\djw{In fact, there has been recent interest in implementing existing sequential tests in online experimentation \citep{miller2015simple}. However, adoption has been limited, because the tests function as ``black boxes'' until the single stopping time when the null hypothesis is accepted or rejected. Inference at general stopping times is not attainable. Thus an off-the-shelf sequential test can provide value for a single experimenter only if the test is specifically optimized to her preferences over power and run-time. On the other hand, by using sequential tests as building blocks to construct always valid p-values and confidence intervals, this paper obtains a real-time interface that can better handle users with heterogeneous preferences.}

\djw{\subsection{Sequential confidence intervals and the LIL}

\label{sec:related_lil}

Our always valid confidence intervals are not the first attempt to construct sequences of intervals which contain an unknown parameter with uniformly high probability over an infinite time-horizon. While our always valid measures emerge from a hypothesis testing framework, such intervals may derived by appealing directly to the Law of the Iterated Logarithm (LIL), which governs the asymptotic behavior of large deviations in sample averages. This approach dates back to \cite{darling1967confidence} and \cite{robbins1970statistical}, where the constructed intervals are referred to as ``confidence sequences". Since then, LIL-type bounds have been obtained for broader classes of (suitably normalized) martingales; for a comprehensive summary, we defer to \cite{pena2008self}. Further, work such as \cite{balsubramani2014sharp} has tightened the LIL at finite samples, while \cite{zhao2016adaptive} extends various classical concentration inequalities to hold at data-dependent stopping times. Independent of our work in this paper, related work has leveraged these LIL-type bounds for the construction of always valid confidence intervals for a range of parametric and non-parametric settings \citep{balsubramani2015sequential, howard2018uniform}.

In Section \ref{sec:comparison_emp}, we compare the empirical performance of our always valid measures derived from the mSPRT against these LIL-based intervals as vehicles for trading off power against run-time. We note that the LIL provides intervals whose size shrinks at a faster asymptotic rate than the mSPRT, and so offers better inference for users who are willing to wait until power very close to one is achieved. For most users, however, our empirical investigation suggests the stronger finite sample performance of the mSPRT would result in a more efficient trade-off. For this reason, we consider it the most appropriate choice for an A/B testing platform.}

\djw{
\subsection{Bandit algorithms}

\rj{While the work described in this paper largely addresses A/B testing carried out using the toolkit of statistical inference (i.e., hypothesis testing), it is worth noting that A/B testing is often used to find the best alternative. For example, a marketing team may be interested in testing two designs for the same web page, with the only goal to maximize conversion rate.  The latter type of problem has been extensively studied in the literature on {\em multi-armed bandits} \citep{bubeck2012regret, lattimore2018bandit}.} \djw{This has led to some use of bandit algorithms in industry in place of hypothesis testing \citep{scott2015multi}.}

\rj{Two variants of the bandit problem formulation are relevant: the {\em pure exploration} problem (when the rewards during the experiment are ignored) \citep{bubeck2009pure}, and the more widely studied {\em regret minimization} problem where the rewards earned during experimentation are also optimized as part of the objective function.  The pure exploration problem is more relevant to the industrial experimentation scenarios that inspired this paper, as in these situations the experimentation period is relatively short compared to the lifespan of any ``winning'' feature that is deployed.  Therefore the rewards earned during experimentation itself can be safely ignored.  In solving the pure exploration problem, the allocation of treatments to incoming traffic is modified dynamically, in order to provide a decision at some minimal stopping time on which treatment is the best.

We note here that much of the interest in sequential confidence intervals has arisen in connection with the literature on pure-exploration bandits, where the target error bounds are of a different flavor than the Type I error rate control sought in this paper. While we seek to control the probability of falsely detecting any treatment effect when no such effect exists, the typical focus in that literature is on {\em probably almost correct} (PAC) bounds \citep{even2002pac, kalyanakrishnan2012pac}. There it is only necessary to identify the winning treatment with high probability in the case that a treatment effect exists and exceeds some pre-specified threshold. For instance, \cite{kaufmann2014complexity} uses such a formulation to characterize the complexity of pure-exploration problems. \cite{jamieson2014lil} and \cite{abbasi2011improved} each present bandit algorithms, which improve regret by leveraging sequential confidence intervals that are always valid in the PAC sense. PAC-style bounds are typically not sufficient for an A/B testing platform, because the threshold on treatment effects sought cannot be tailored to the heterogeneous users.

Finally, in light of this discussion on bandits, it is worth noting why our own work emphasizes the hypothesis testing viewpoint as opposed to a pure-exploration multi-armed bandit approach.  The main point we make here is that at the same time as the best alternative is sought prospectively, there is also a post-experiment interest in {\em inference}; in particular, the experimenter often wants a confidence interval on the effect size of a variation that does not win.  Part of the reason is operational: if the gain of a winning variation is insufficient over the status quo, the deployment cost may not be worthwhile.  Such confidence intervals are also strategically important, as they provide guidance on what types of experiments may be worth trying in the future.  Our observation is that the hypothesis testing framework remains the dominant form of A/B testing in industrial deployment, at least partly for such reasons, and for this reason we have focused our attention on this setting. However, bandit methods are also practically valuable, and developing methods for inference with adaptive allocation remain interesting directions for future work.  We return to this point in our conclusion, Section \ref{sec:conclusion}.}

\subsection{Sequential multiple hypothesis testing}}

There has also been recent interest in achieving multiple hypothesis testing controls in sequential contexts.  For the most part, work in this area considers a different form of streaming data to the one described in this paper: \djw{\cite{demets1994interim}, \cite{foster2008alpha} and \cite{javanmard2016online} provide so-called ``$\alpha$-spending'' and ``$\alpha$-investing" methods to control the family-wise error rate (FWER) or the false discovery rate (FDR) when experiments are performed sequentially, but within each experiment the data is accessed only once.}

\djw{However, recent work has started to address the within-experiment data arrival process that characterizes A/B testing. \cite{yang2017framework} combines $\alpha$-investing with sequential hypothesis testing to enable FDR control in this regime. \cite{jamieson2018bandit} goes a step further, allocating traffic to treatments dynamically with the goal of achieving statistical significance quickly, while still bounding FDR. \cite{malek2017sequential} investigates when always valid p-values can be used to achieve other multiple testing bounds beyond FWER and FDR.}


\section{Preliminaries}
\label{sec:preliminaries}

To begin, we suppose that our data can be modeled as independent observations from an exponential family $\v{X} = (X_n)_{n =1}^\infty \stackrel{iid}{\sim} F_\theta$, where the parameter $\theta$ takes values in $\Theta \subset \reals^p$. Throughout the paper, $(\mc{F}_n)_{n = 1}^\infty$ will denote the filtration generated by $(X_n)_{n =1}^\infty$ and $\P_\theta$ will denote the measure (on any space) induced under the parameter $\theta$. Our focus is on testing a simple null hypothesis $H_0: \theta = \theta_0$ against the composite alternative $H_1: \theta \neq \theta_0$.  (In Section \ref{sec:abtests} we adapt our analysis to two-sample hypothesis testing, as is needed to test differences between control and treatment in an A/B test.)

{\bf Decision rules}.  In general, a decision rule is a pair $(T, \delta)$, where $T$ is a (possibly infinite) stopping time for $(\mc{F}_n)_{n = 1}^\infty$ that denotes the sample size at which the test is ended, and $\delta$ is a binary-valued, $(\mc{F}_T)$-measurable random variable, where $\delta = 1$ indicates that $H_0$ is rejected; note that $\delta = 0$ must hold a.s.~if $T = \infty$.  Note that we allow the possibility that the decision rule can be data-dependent; when $T$ is not data-dependent, we refer to the rule as a {\em fixed horizon} decision rule.

{\bf Type I error}.  Type I error is the probability of erroneous rejection under the null, i.e., $\P_{\theta_0}(\delta = 1)$.  We assume that the user wants to bound Type I error at level $\alpha \in (0,1)$.

{\bf Sequential tests}.  Given $\alpha$, we typically consider a family of decision rules parameterized by $\alpha$.  Formally, a {\em sequential test} is a family of decision rules $(T(\alpha), \delta(\alpha))$ indexed by $\alpha \in (0,1)$ such that:
\begin{enumerate}
\item The decision rules are {\em nested}: $T(\alpha)$ is a.s.~nonincreasing in $\alpha$, and $\delta(\alpha)$ is a.s.~nondecreasing in $\alpha$.
\item For each $\alpha$, the Type I error is bounded by $\alpha$: $\P_{\theta_0}(\delta = 1) \leq \alpha$.
\end{enumerate}
Note that sequential tests allow the possibility that the decision rules are data-dependent, though strictly speaking fixed horizon decision rules are allowed in this definition as well.


{\bf Fixed horizon testing}.  Under the default fixed horizon testing approach, we restrict to decision rules $(n,\delta)$, where the stopping time is required to be deterministic.  In this setting, the objective is to maximize the power (the probability of detection under $H_1$) at that $n$.  Indeed, for data in an exponential family, for any given $n$, there exist a family of uniformly most powerful (UMP) tests parameterized by $\alpha$, each of which maximizes power uniformly over $\theta$ among tests with Type I error rate $\alpha$.  These tests reject the null if a particular test statistic $\tau_n$ exceeds a threshold $k(\alpha)$ \citep[see, e.g., Chapter 4 of ][]{lehmann1986testing}.

While the tests maximize power for the given $n$, the power increases as $n$ is increased, and so the user must choose $n$ to trade off power against the opportunity cost of waiting for more samples.  The challenge for the user is that the power is a steep function of the true $\theta$, so good advance knowledge on the size of the effect sought is required.

{\bf The fixed horizon user interaction model}.  Testing platforms typically allow users to implement their optimal test via {\em p-values}.  Specifically, the p-value at time $n$ corresponding to the UMP test is:
\[ p_n = \inf \{ \alpha : \tau_n \geq k(\alpha) \}. \]
In other words, this p-value is the {\em smallest $\alpha$ such that the $\alpha$-level test with sample size $n$ rejects $H_0$.}

The process $p_n$ provides sufficient information for the user to implement her desired test with ease: she waits for her chosen $n$, and rejects the null hypothesis if $p_n \leq \alpha$.  In addition, $p_n$ ensures transparency in the following sense: since each rule $\delta_n(\alpha)$ controls Type I error at level $\alpha$, any other user can threshold the p-value obtained at her own appropriate $\tilde{\alpha}$ level to satisfy her desired Type I error bound.

In fact, to control Type I error, we require only that the p-value is {\em super-uniform}:
\begin{equation}
\label{eq:fh_pvalues}
\forall s \in [0,1], \,\, \P_0(p_n \leq s) \leq s.
\end{equation}
More generally, we refer to any $[0,1]$-valued, $(\mc{F}_n)$-measurable random variable $p_n$ that satisfies \eqref{eq:fh_pvalues} as a {\em fixed horizon p-value} process for the choice of sample size $n$.

Confidence intervals can be constructed from the tests $\delta_n(\alpha)$ associated with fixed horizon p-values for $H_0 : \theta = \tilde{\theta}$ at each $\tilde{\theta} \in \Theta$ by considering the set of $\theta$ that are not rejected.  What matters is the following coverage bound: a {\em $(1 - \alpha)$-level fixed horizon confidence interval} is any $(\mc{F}_n)$-measurable random set $\CI_n \subset \Theta$ where
\begin{equation}
\forall \theta \in \Theta, \,\, \P_\theta (\theta \in \CI_n) \geq 1 - \alpha.
\end{equation}

\section{Always valid inference}
\label{sec:always_valid}

Our goal is to let the user stop the test whenever they want, in order to trade off power with run-time as they see fit; the p-value they obtain should control Type I error.  Our first contribution is the definition of {\em always valid} p-values as those processes that achieve this control.

\begin{definition}
\label{def:av}
A sequence of fixed horizon p-values $(p_n)$ is an {\em always valid p-value} process if given any (possibly infinite) stopping time $T$ with respect to $(\mc{F}_n)$, there holds:
\begin{equation}
\label{eq:avpval}
\forall s \in [0,1], \,\, \P_{\theta_0} (p_T \leq s) \leq s.
\end{equation}
\end{definition}

The following theorem demonstrates that always valid p-values are in a natural correspondence with sequential tests.

\begin{theorem}
\label{thm:av}
\begin{enumerate}
\item Let $(T(\alpha), \delta(\alpha))$ be a sequential test.  Then
$$ p_n = \inf \{ \alpha : T(\alpha) \leq n, \delta(\alpha) = 1 \} $$
defines an always valid p-value process.
\item For any always valid p-value process $(p_n)_{n=1}^\infty$, a sequential test $(\tilde{T}(\alpha), \tilde{\delta}(\alpha))$ is obtained from $(p_n)_{n=1}^{\infty}$ as follows:
\begin{align}
\tilde{T}(\alpha) &= \inf \{ n : p_n \leq \alpha \}; \label{eq:T_alpha} \\
\tilde{\delta}(\alpha) &= \1 \{ \tilde{T}(\alpha) < \infty\}. \label{eq:delta_alpha}
\end{align}
If $(p_n)_{n=1}^\infty$ was derived as in part (1) and $T = \infty$ whenever $\delta = 0$, then $(\tilde{T}(\alpha), \tilde{\delta}(\alpha)) = (T(\alpha), \delta(\alpha))$.
\end{enumerate}
\end{theorem}

\begin{proof}{\rj{Proof of Theorem \ref{thm:av}.}}
For the first result, nestedness implies the following identity for any $s \in [0,1], n \geq 1, \epsilon > 0$:
$$ \{p_n \leq s\} \subset \{T(s + \epsilon) \leq n, \delta(s + \epsilon) = 1\} \subset \{\delta(s + \epsilon) = 1\}. $$
Therefore:
$$ \P_{\theta_0} (p_T \leq s) \leq \P_{\theta_0} ( \cup_n \{p_n \leq s\}) \leq \P_{\theta_0} (\delta(s + \epsilon) = 1) \leq s + \epsilon. $$
The result follows on letting $\epsilon \to 0$. For the converse, it is immediate from the definition that the tests are nested and $\delta(\alpha) = 0$ whenever $T(\alpha) = \infty$.  For any $\epsilon > 0$
\begin{equation} \nonumber \begin{split} \P_{\theta_0} (\delta(\alpha) = 1) = \P_{\theta_0} (T(\alpha) < \infty) &\leq \P_{\theta_0} (p_{T(\alpha)} \leq \alpha + \epsilon) \leq \alpha + \epsilon \end{split} \end{equation}
where the last inequality follows from the definition of always validity. Again the result follows on letting $\epsilon \to 0$.\hfill$\Box$
\end{proof}

The p-value defined in part (1) of the theorem is not the unique always valid p-value associated with that family of sequential tests (i.e., for which part (2) holds). However, among such always valid p-values it is a.s.~{\em minimal} at every $n$, resulting from the fact that it is a.s.~{\em monotonically non-increasing} in $n$.  Thus we have a one-to-one correspondence between monotone always valid p-value processes and families of sequential tests that do not give up for failure; i.e., where $\delta = 0$ implies $T = \infty$. These processes can be seen as the natural representation of those sequential tests in a streaming p-value format.

{\bf The new user interaction model.} The time $\tilde{T}(\alpha)$ represents the natural stopping time of a hypothetical user who incurs no opportunity cost from longer experiments. By thresholding the p-value at $\alpha$ at that time, she recovers the underlying sequential test and is able to reject $H_0$ whenever $\delta = 0$. Of course, a real user cannot wait forever, so she must stop the test and threshold the p-value at some potentially earlier, a.s.~finite stopping time. In so doing, she sacrifices some detection power. This trade-off for the user between power and average run-time is a central concern of our proposed design, and is studied in more detail in Section \ref{sec:optimality}.

{\bf Confidence intervals}. Always valid CIs are defined analogously and may be constructed from always valid p-values just as in the fixed horizon context. Proposition \ref{prop:pval_CI_duality} follows immediately from the definitions.

\begin{definition}
A sequence of fixed-horizon $(1-\alpha)$-level confidence intervals $(\CI_n)$ is an {\em always valid \djw{$(1-\alpha)$-level} confidence interval process} if, given any stopping time $T$ with respect to $(\mc{F}_n)$, there holds:
\begin{equation}
\label{eq:avCI}
\forall \theta \in \Theta, \,\, \P_\theta(\theta \in \CI_T) \geq 1 - \alpha.
\end{equation}
\end{definition}

\begin{proposition}
\label{prop:pval_CI_duality}
Suppose that, for each $\tilde{\theta} \in \Theta$, $(p_n^{\tilde{\theta}})$ is an always valid p-value process for the test of $\theta = \tilde{\theta}$.  Then
\[ \CI_n = \left \{ \theta : p_n^{\theta} > \alpha \right \}  \]
is an always valid $(1-\alpha)$-level CI process.
\end{proposition}

\section{\rj{Efficient always valid p-values via the mSPRT}}
\label{sec:optimality}

\rj{As noted in the preceding section, users who continuously monitor experiments are making a dynamic tradeoff between two objectives: detecting true effects with maximum probability and minimizing the typical length of experiments.  A significant challenge for the platform is that always valid p-values must be designed without prior knowledge of the user's preferences.  We are led, therefore, to consider the following design problem: how should always valid p-values be designed to lead users to an efficient tradeoff between power and run-length, without access to the user's preferences in advance?

In this section we first introduce a natural model of user behavior that encodes a tradeoff between power and run-length, where users are characterized by their Type I error tolerance, $\alpha$, and the maximum number of observations $M$ they can afford.  We assume that such a user stops and rejects $H_0$ at either the first time the always valid p-value crosses $\alpha$, or at time $M$, whichever comes first.  We refer to such a user as a {\em type $(M, \alpha)$ user.}  We study efficiency of an always valid p-value process via the {\em power profile} and {\em relative run-length profile} that the process induces for an $(M, \alpha)$ user across possible treatment effects; the former gives the probability that true effects are detected by this user, and the latter gives the expected run-length for such a user, normalized by $M$.  Informally, our goal is to deliver high power at low relative run-lengths.

To achieve this goal, we consider always valid p-value processes derived from a particular family of sequential tests, the {\em mixture sequential probability ratio tests} (mSPRT).  The mSPRT stops the first time a mixture (over treatment effects) of likelihood ratios against the null crosses a threshold.  These were first introduced by \cite{robbins1970statistical}. The mSPRT provides an easily implemented family of always valid p-value processes, as we discuss further in Section~\ref{sec:abtests}.  In this section, we focus on the efficiency properties of the mSPRT, by comparing it to feasible decision rules induced by other sequential tests.

We begin our study in Section \ref{sec:firstordereff} by considering {\em first-order efficiency} of the mSPRT, in the limit where $\alpha \to 0$.  We first note that for users where $M$ is relatively large or relatively small compared to $\log(1/\alpha)$, efficiency (in an appropriate sense) is a relatively weak requirement, and easily established in particular for the mSPRT.  Thus the more interesting analysis is for users where $M \sim \log(1/\alpha)$.  For such users, we establish that the mSPRT satisfies a desirable first-order efficiency property: there is no other feasible decision rule that yields a relative run-length that is lower on some non-null effects, while meeting the size constraint $\alpha$ and yielding higher power at all non-null effects.

First-order efficiency is important, but does not tell the complete story.  In particular, the preceding efficiency result does not depend on the mixing distribution employed by the mSPRT, whereas we should certainly expect that finite $M$ (or fixed $\alpha$) performance will be influenced by the choice of the mixing distribution.  In Section \ref{sec:Hdependence} we theoretically investigate the importance of the choice of prior.  We first theoretically investigate this question by considering a Bayesian setting where effects are drawn from a prior $G$.  We characterize the expected run-length minimizing mixture for type $(M,\alpha)$ users where $M \sim \log(1/\alpha)$.  Specialized to the setting of normal data, we note the intuitive result that this optimal mixture involves ``matching'' the prior in an appropriate sense.  In Section \ref{sec:Hdependence_emp}, we carry out an empirical investigation of the role of the choice of mixing distribution for $(M, \alpha)$ users; while we find that the choice of mixture matters, we also find that the performance of the mSPRT is surprisingly robust to misspecification.

We conclude our analysis by comparing the mSPRT to other decision rules.  First, in Section \ref{sec:fixedhorizon} we compare the mSPRT to fixed-horizon testing.  In particular, we show that for $(M, \alpha)$ users, the mSPRT provides an improvement over fixed-horizon testing in the $\alpha \to 0$ limit, even when the fixed-horizon test is tailored to prior knowledge.  In Section \ref{sec:fixedhorizon_emp}, we complement this theoretical result with empirical comparison of the mSPRT to fixed-horizon testing, demonstrating the practical benefits for $(M, \alpha)$ users.  Second, we note the mSPRT is only one of many possible choices of sequential tests in the literature that satisfy desirable optimality criteria.  We conclude our study of efficiency of the mSPRT in Section \ref{sec:comparison_emp}, where we empirically compare its performance to other always valid p-value processes derived from sequential tests (including the LIL approach discussed in Section \ref{sec:related}).  These empirical results demonstrate that the mSPRT delivers high performance in the regime of $M$ and $\alpha$ relevant for practical application.

The section is organized as follows.  First, we present our theoretical results on first-order efficiency (Section \ref{sec:firstordereff}), choice of mixing distribution (Section \ref{sec:Hdependence}), and comparison to fixed-horizon testing (Section \ref{sec:fixedhorizon}).  Next, in Section \ref{sec:empirical}, we present our empirical results in a simulation set up involving a single stream of normal data, and present an empirical investigation of the dependence on the prior (Section \ref{sec:Hdependence_emp}), an empirical comparison of the mSPRT to fixed-horizon testing (Section \ref{sec:fixedhorizon_emp}), and a comparison to other sequential testing approaches (Section \ref{sec:comparison_emp}).}



\subsection{The user model}

Of course, any specific user's preferences will be highly nuanced.  In our analysis, for technical simplicity we consider the following user model: we assume user preferences can be characterized by a parameter $M$ representing the maximum number of observations that the user is willing to wait, together with their Type I error tolerance $\alpha$.  Given an always valid p-value process, we consider a simple model of user behavior: users stop at either the first time the p-value drops below $\alpha$ (in which case they reject the null), or at time $M$ (in which case they do not reject the null unless the p-value at time $M$ is also below $\alpha$), whichever occurs first.

We refer to such a user as a {\em $(M,\alpha)$ user}.  In the remainder of the section, our goal will be to make near-optimal tradeoffs between power and run-length for users in the limit where $\alpha \to 0$, without prior knowledge of their preferences, $(M,\alpha)$.  \lp{ Given the equivalence to a sequential test  $(T(\alpha), \delta(\alpha))$, we define the $(M,\alpha)$ user's decision rule by $T(M, \alpha) \triangleq \min\{ T(\alpha), M \}$, and remark that $\delta(M, \alpha) = 1$ if and only if $T(\alpha) \leq M$. }


\subsection{The mSPRT}

The always valid p-values we employ are derived from a particular family of sequential tests: the {\em mixture sequential probability ratio test} (mSPRT) \citep{robbins1970statistical}.  We begin by imposing slight restrictions on the data model: we assume that the data is real valued and drawn from a single parameter exponential family, $f_\theta(x) = F'_\theta(x) = f_0(x) \exp(\theta x - \psi(\theta))$, where tests are of the natural parameter $\theta$, $\Theta$ is an open interval, $\psi^{''}(\theta) > 0$ for all $\theta$, and $\E|X_1|^4 < \infty$.  \rj{The function $\psi(\theta)$ is referred to as the {\em log partition function} for the family.}

The mSPRT is parameterized by a {\em mixing distribution} $H$ over $\Theta$, which we restrict to have everywhere continuous and positive derivative. Given an observed sample average $s_n$ up to time $n$, the likelihood ratio of $\theta$ against $\theta_0$ is $(f_\theta(s_n)/f_{\theta_0}(s_n))^n$.  Thus we define the {\em mixture likelihood ratio} with respect to $H$ as
\begin{equation} \label{eq:msprt_definition} \Lambda_n^H(s_n) = \int_{\Theta} \left( \frac{f_\theta(s_n)}{f_{\theta_0}(s_n)}\right)^n dH(\theta).
\end{equation}
The mSPRT is then defined by:
\begin{align}
T^H(\alpha) & = \inf\{n : \Lambda_n^H(S_n) \geq \alpha^{-1} \}; \ \ \text{and} \label{eq:mSPRT1}\\
\delta^H(\alpha) &= \1 \left\{T^H(\alpha) < \infty\right\}, \label{eq:mSPRT2}
\end{align}
 \rj{where $S_n = \sum_{i=1}^n X_i/n$}.  The choice of threshold $\alpha^{-1}$ on the likelihood ratio ensures Type I error is controlled at level $\alpha$, via standard martingale techniques \citep{siegmund1985sequential}. Intuitively, $\Lambda_n^H(S_n)$ represents the evidence against $H_0$ in favor of a mixture of alternative hypotheses, based on the first $n$ observations. The test rejects $H_0$ if the accumulated evidence ever becomes large enough.
%

Our first motivation for considering mSPRTs is that they are {\em tests of power one} \citep{RS74}: for all $\alpha$ and $\theta \neq \theta_0$, there holds:,
\[  \P_\theta( T(\alpha) < \infty, \delta(\alpha) = 1) = 1. \]
In other words, for the hypothetical user who can wait forever, any mSPRT delivers power one for any alternative. \djw{Second, mSPRTs have been studied from a decision-theoretic framework, where the cost of longer experiments is taken to be the terminal sample size multiplied by some cost $c$ per observation. The goal there is to balance this penalty against the costs associated with Type I and Type II errors. mSPRTs are found to be asymptotically optimal as $c \to 0$ \citep{lai2001sequential}. This is a large data limit, analogous to the $\alpha \to 0, \, M \to \infty$ setup of this paper.}


For later reference, we note the following result due to \cite{PS1975}: for any mixing distribution $H$, as $\alpha \to 0$,
\begin{equation}
\label{eq:limit_prob}
T^H(\alpha) / \log(1/\alpha) \to I(\theta, \theta_0)^{-1} := \left \{ (\theta - \theta_0) \psi'(\theta) - (\psi(\theta) - \psi(\theta_0)) \right \}^{-1}
\end{equation}
holds in probability and in $\mathcal{L}^2$, where $\psi(\theta)$ is the log-partition function for the family $F_\theta$.  This result characterizes the run-length of the mSPRT in the small $\alpha$ limit, and plays a key role in our subsequent study of efficiency.

\subsection{\rj{First-order efficiency in the $\alpha \to 0$ limit}}
\label{sec:firstordereff}

We now formalize our study of the power and run-length tradeoff for an $(M, \alpha)$ user.  \lp{In this section, the set of $(T(M, \alpha), \delta(M, \alpha))$ decision rules is referred to as the set of {\em feasible} decision rules for an $(M, \alpha)$ user.}


We first map the two objectives of the user to formal quantities of interest.  First, an $(M, \alpha)$ user will want to choose her decision rule to maximize the {\em power profile} $\nu(\theta) = \P_\theta(\delta = 1)$ over $\theta \neq \theta_0$. Second, she will want to minimize the {\em relative run-length profile}, i.e., the run-length measured against the maximum available to her, $\rho(\theta) = \E_\theta(T) / M$, viewed as a function of $\theta$.

{\em Perfect efficiency} would entail $\rho(\theta) = 0$ and $\nu(\theta) = 1$ for all $\theta \neq \theta_0$.  Of course, perfect efficiency is generally unattainable for feasible decision rules.  In this section we study the best achievable performance a user can hope for, in the limit where $\alpha \to 0$.

Our analysis depends on the characterization of run-length of the mSPRT in \eqref{eq:limit_prob}.  The consequence is that if we produce always valid p-values using the mSPRT, then in the limit as $\alpha \to 0$ the study of efficiency divides into three cases depending on the relative values of $M$ and $\log(1/\alpha)$.  We consider these cases in turn.

{\bf ``Aggressive'' users: $M \gg \log(1 / \alpha)$.} Users in this regime are aggressive; $\alpha$ is large relative to the maximum run-length they have chosen.  In this regime, any mSPRT asymptotically recovers perfect efficiency in the limit where $\alpha$ \djw{is} small.  Intuitively, because the user is willing to wait a substantially longer time than $\log(1/\alpha)$, a sublinear fraction of their maximum run-length is used by the mSPRT; since the mSPRT is a test of power 1, this means the user receives power near 100\% in return for a near-zero run-length profile.  The proof of the following result follows immediately from (\ref{eq:limit_prob}).
\begin{proposition}
	Given any mixture $H$, let $\rho(\theta)$ and $\nu(\theta)$ be the relative run-length and power profiles, respectively, associated with $(T^H(M, \alpha), \delta^H(M, \alpha))$. If $\alpha \to 0$ and $M \to \infty$ such that $M/\log(1/\alpha) \to \infty$, we have $\rho(\theta) \to 0$ and $\nu(\theta) \to 1$ at each $\theta \neq \theta_0$.
\end{proposition}

{\bf ``Conservative'' users: $M \ll \log(1 / \alpha)$.}  Users in this regime are conservative; $\alpha$ is small relative to the maximum run-length they have chosen.  In this case, experimentation is not productive: the user is unwilling to wait long enough to detect any effects.  Thus any mSPRT trivially performs as well as any feasible decision rule for the user.
\begin{proposition}
\label{prop:conservative}
	For each $(M, \alpha)$, and any feasible decision rule, let $\nu$ be the associated power profile. Then if $\alpha \to 0, M \to \infty$ such that $M/\log(1/\alpha) \to 0$, we have $\nu(\theta) \to 0$ for each $\theta$.
\end{proposition}

{\bf ``Goldilocks'' users: $M \sim \log(1 /\alpha)$.}  This is the interesting case, where experimentation is worthwhile but statistical analysis is non-trivial.  To proceed we require an additional definition.  For a family of sequential tests, we want to define a measure of the worst-case efficiency over $\theta \neq \theta_0$ for an $(M, \alpha)$ user.  Informally, we define this as the relative efficiency of the truncated test they obtain by minimizing the relative run-length everywhere, compared with any other test that offers at least as good power everywhere.  This is formalized as follows.

\begin{definition}
	Given a sequential test $(T(\alpha), \delta(\alpha))$, let $\rho(\theta; \alpha, M)$ and $\nu(\theta; \alpha, M)$ be the relative run-length and power profiles associated with $(T(M, \alpha), \delta(M, \alpha))$. The {\em relative efficiency} of this test at $(M, \alpha)$ is
\[ \phi(M, \alpha) = \inf_{(T, \delta)} \inf_{\theta \neq \theta_0} \frac{\rho(\theta)}{\rho(\theta; \alpha,M)} \]
where the infimum is taken over all weakly more powerful, feasible decision rules: i.e., over tests such that $T \leq M$ a.s.~, $\P_{\theta_0}(\delta = 1) \leq \alpha$), and for all $\theta \neq \theta_0$ there holds $\nu(\theta) \geq \nu(\theta; \alpha, M)$.
\end{definition}

Our main result demonstrates that in the regime where $M \sim \log(1/\alpha)$, any mSPRT has relative efficiency approaching unity when $\alpha \to 0$.

\begin{theorem}
\label{thm:efficiency}
	Suppose $\psi''$ is absolutely continuous and there is an open interval around $\theta_0$ where $\psi''(\theta) < \infty$. Given any $H$, let $\phi(M, \alpha)$ be the efficiency of the mSPRT $(T^H(\alpha), \delta^H(\alpha))$. If $\alpha \to 0, M \to \infty$ such that $M = O(\log(\alpha^{-1}))$, we have $\phi(M, \alpha) \to 1$.
\end{theorem}

\rj{Note that this result is not dependent on the mixing distribution $H$.  This is a consequence of the fact that we study efficiency only to first-order in the limit as $\alpha \to 0$.  However, as we find in the next section, the choice of prior can have an important second order effect on performance.}

\subsection{\rj{The role of the mixing distribution $H$}}
\label{sec:Hdependence}

\rj{In this section we investigate the impact of the choice of the mixing distribution $H$ on performance.
As noted in the preceding section, it is users with $M \sim \log(1/\alpha)$ where there is a meaningful tradeoff between run-length and power; therefore we focus our attention on the role of $H$ for this class of users.  Theorem \ref{thm:efficiency} establishes that for these users any mSPRT is asymptotically efficient as $\alpha \to 0$, i.e., that no other test can offer a uniform improvement at every $\theta$ to first order in that limit.  However, the choice of $H$ does influence the tradeoff between performance at different effect sizes, and this tradeoff will yield a relevant second order effect on performance.

In this subsection, we consider a Bayesian setting where we have a prior $\theta \sim G$ under $H_1$.  The following theorem gives the mixing distribution $H$ that minimizes the relative run-length on average over this prior for a ``Goldilocks'' user with parameters $(M, \alpha)$ (i.e., a user with $M \sim \log(1/\alpha)$).}


\begin{theorem} \label{thm:optimal_mixture}
Suppose $G$ is absolutely continuous with respect to Lebesgue measure, and let $H_\gamma$ be a parametric family, $\gamma \in \Gamma$, with density $h_\gamma$ positive and continuous on $\Theta$. Let $\rho_\gamma(\theta)$ be the profile of relative run-lengths associated with $(T^{H_\gamma}(M, \alpha), \delta^{H_\gamma}(M, \alpha))$. Then up to $o(1)$ terms as $\alpha \rightarrow 0$ and with $M = O(\log(1/\alpha))$, the average relative run-length $\E_{\theta \sim G} \{\rho_\gamma(\theta)\}$ is minimized by any $\gamma^*$ such that:
\begin{equation} \label{eq:prior_matching} \gamma^* \in \argmin_{\gamma \in \Gamma} - \E_{\theta \sim G} \v{1}_{A(M, \alpha)} I(\theta, \theta_0)^{-1} \log h_\gamma(\theta), \end{equation}

where $A(M, \alpha) = \left \{ \theta : I(\theta, \theta_0) \geq\log (1/\alpha) / M \right \}$.

If $h_\gamma(\theta) = q(\gamma_1,\gamma_2) e^{\gamma_1 \theta - \gamma_2 \psi(\theta)}$ \djw{is} a conjugate prior for $f_\theta$, the data distribution, then $\gamma^*$ solves:

\[ \frac{ \partial q(\gamma_1,\gamma_2) / \partial \gamma_1}{\partial q(\gamma_1, \gamma_2) / \partial \gamma_2} = \frac{ \E_{\theta \sim G} \v{1}_A \theta I(\theta, \theta_0)^{-1}}{\E_{\theta \sim G} \v{1}_A \psi(\theta) I(\theta, \theta_0)^{-1}}. \]

\end{theorem}

\rj{We remark here that, consistent with our finding of first-order optimality for any choice of $H$ in Theorem \ref{thm:efficiency}, it follows from our proof that the choice of $\gamma$ does not impact  $\E_{\theta \sim G} \{\rho_\gamma(\theta)\}$ to first order in the limiting regime of Theorem \ref{thm:optimal_mixture}.}

Heuristically, the mSPRT rejects $H_0$ when there is sufficient evidence in favor of any $H_1 : \theta \neq \theta_0$, weighted by the distribution $H$ over alternatives. Thus we might expect optimal sampling efficiency when $H$ is {\em matched} to the prior. If we specialize to normal data, $f_\theta(x) = \phi(x-\theta)$, a centered normal prior, $G(\theta) = \Phi(\frac{\theta}{\tau})$, and consider normal mixtures, $h_\gamma(\theta) = \frac{1}{\gamma} \phi(\frac{\theta}{\gamma})$, this intuition is mostly accurate.  In that case, the optimal choice of mixing variance becomes:
\begin{equation}  \label{eq:prior_matching_normal} \gamma^{2*} = \tau^2 \frac{ \Phi(-b) }{ \frac{1}{b} \phi(b) - \Phi(-b)}
\end{equation}
for $b = \left( \frac{2 \log \alpha^{-1} }{M \tau^2} \right)^{1/2}$, equal to the prior variance multiplied by a factor correcting for anticipated truncation. Sampling efficiency is improved by weighting towards larger effects when few samples are available and smaller effects where there is ample data.

\rj{We subsequently investigate the importance of the choice of $H$ empirically in Section \ref{sec:Hdependence_emp} below.}

\subsection{\rj{Comparison to fixed-horizon testing}}
\label{sec:fixedhorizon}

\rj{We now theoretically compare decisions based on mSPRT p-values with fixed-horizon testing.  In general, the fixed-horizon sample size must be chosen in reference to the effect sizes where detection is needed; therefore, we compare performance of the mSPRT to a fixed-horizon test that is calibrated to have good average power over the prior $G$.}

For convenience, we focus on the normal case $G = N(0,\tau^2)$, and fix $H = N(0,(\gamma^*)^2)$.  We consider two rival tests for an arbitrary ``Goldilocks'' user: the mSPRT truncated at $M$, and the fixed-horizon test chosen to have the same average power over this prior.  For $\alpha$ small, the former has lower average relative run-length over $G$, as formalized in the following proposition.
\begin{proposition} \label{prop:improvement}
\rj{Suppose $G = N(0, \tau^2)$ and $H = N(0, (\gamma^*)^2)$ (cf.~Theorem \ref{thm:optimal_mixture}), and let $\nu^*(\theta; \alpha, M)$ and $\rho^*(\theta; \alpha, M)$ be the power profile and relative run-length profile of the resulting mSPRT truncated at $M$.  Let $n^*$ be the sample size such that the UMP fixed-horizon test at sample size $n^*$ has expected power matching the truncated mSPRT, i.e., equal to $\E_{\theta \sim G}[ \nu^*(\theta; \alpha, M)]$.  Let $\rho_f(\theta; \alpha, M)$ be the relative run-length profile of this fixed horizon test.

If $\alpha \to 0, M \to \infty$, such that $M = O(\log(1/\alpha))$, then $\E_{\theta\sim G}[\rho^*(\theta; M,\alpha)]/\E_{\theta \sim G}[\rho_f(\theta; \alpha, M)] \to 0$.}
\end{proposition}

\rj{We subsequently empirically compare the mSPRT decision rule to fixed-horizon testing with single stream data in Section \ref{sec:fixedhorizon_emp} below.  We also make this comparison again in Section \ref{sec:realworld}, using experiment data from the large-scale commercial A/B testing platform where these methods were deployed.}

\lp{
\subsection{Empirical analysis}
\label{sec:empirical}

In this subsection we complement our preceding theoretical analysis with a number of empirical analyses.  For all our simulations, we assume the normal data, prior, and mixture setup leading to \eqref{eq:prior_matching_normal}: i.e., normal data with $f_\theta(x) = \phi(x-\theta)$, a centered normal prior with $G(\theta) = \Phi(\frac{\theta}{\tau})$, and normal mixing distribution with $h_\gamma(\theta) = \frac{1}{\gamma} \phi(\frac{\theta}{\gamma})$.  Informally, our simulations evaluate performance of different decision rules for $(M, \alpha)$ users, on average over effect sizes drawn from the prior.  Formally, for each tuple $(\alpha, M, \tau)$, we draw effect sizes $\theta_b \sim G$, $b = 1,\dots,B$, and we simulate a stream of data $\v{X}_b = (X_n)_{n=1}^{M} \stackrel{iid}{\sim} F_{\theta_b}$.  We estimate average power and run-length profiles by $\hat{\nu} = \frac{1}{B} \sum_{b=1}^B \mathbf{1}_{\{T_b \leq M\}}$ and $\hat{\rho} =  \frac{1}{BM} \sum_{b=1}^B \min\{ T_b, M\}$, where $T_b$ corresponds to the first rejection time of test $T(\alpha)$ for $H_0: \theta = 0$ with data $\v{X}_b$.  Finally, we compare average power $\hat{\nu}$ and average run-length $\hat{\rho}$ across a variety of tests.

We select $(\alpha, M, \tau)$ over the grid generated by outer product of $\alpha \in \{1e^{-4}, 1e^{-2}, 1e^{-1}\}$, $\tau^2 \in \{1e^{-4}, 1e^{-2}, 1e^{-1}\}$, $M \in \{ 1e^1, \dots, 1e^7\}$, and $B = 1e^4$. Rough estimates suggest these are enough Monte Carlo simulations for low variability. The effect size distributions were designed to approximate reality for online experimenters, as $\tau^2= 1e^{-4}$ roughly matches a 10\% relative improvement on 1\% conversion rate, $1e^{-2}$ a 100\% improvement, and $1e^{-1}$ a 1000\% improvement.

To achieve comparison across the wide range of parameters, we present results relative to the re-parameterized maximum run-length $\tilde{M}$, defined as follows:
\[ \tilde{M} = M \left( \frac{ \log \alpha^{-1} }{\tau^2} \right)^{-1}. \]

\subsubsection{The role of the mixing distribution $H$: Empirics}
\label{sec:Hdependence_emp}

Nominally, the dependence of $\gamma^*$ on $M$ presents a challenge. However, as we now demonstrate, the choice of mixing distribution is quite robust with respect to variation in $M$, though getting $\gamma$ ``in the ballpark'' is important.

We simulate the mSPRT for five different mixing regimes, $\gamma \in \{ \gamma^*, 1, 1e^{-1} \tau, \tau, 1e^1 \tau \}$. Figure \ref{fig:gamma_varying_sims} depicts results. Across all cases, we find that $\gamma$ misspecification of around one order of magnitude leads to a less than 5\% drop in average power, and no more than a 10\% increase in average run-length. Missing $\gamma^*$ by two orders of magnitude, however, can result in a 20\% drop in average power and a 40\% increase in average run-length. Note that we do not show results for severely under-powered parameter combinations ($\hat{\nu} < 0.1$), as the variance in estimates distracts from the overall picture.

The story does change for ``non-Goldilocks'' users. As $\tilde{M}$ grows, $\hat{\nu} \rightarrow 1$ regardless of $\gamma$, resulting in muted gains from $H$ optimization, while $\hat{\rho}$ shows little sensitivity to $\tilde{M}$. Lastly, we remark on the few cases where $\gamma \neq \gamma^*$ achieves superior results. These all have $\alpha = 0.1$ and $\tilde{M} < 1.0$, indicating a horizon for the breakdown of asymptotics used in Theorem \ref{thm:optimal_mixture} for finite sample regimes.

\begin{figure*}
\begin{center}
\includegraphics[width=0.9\textwidth]{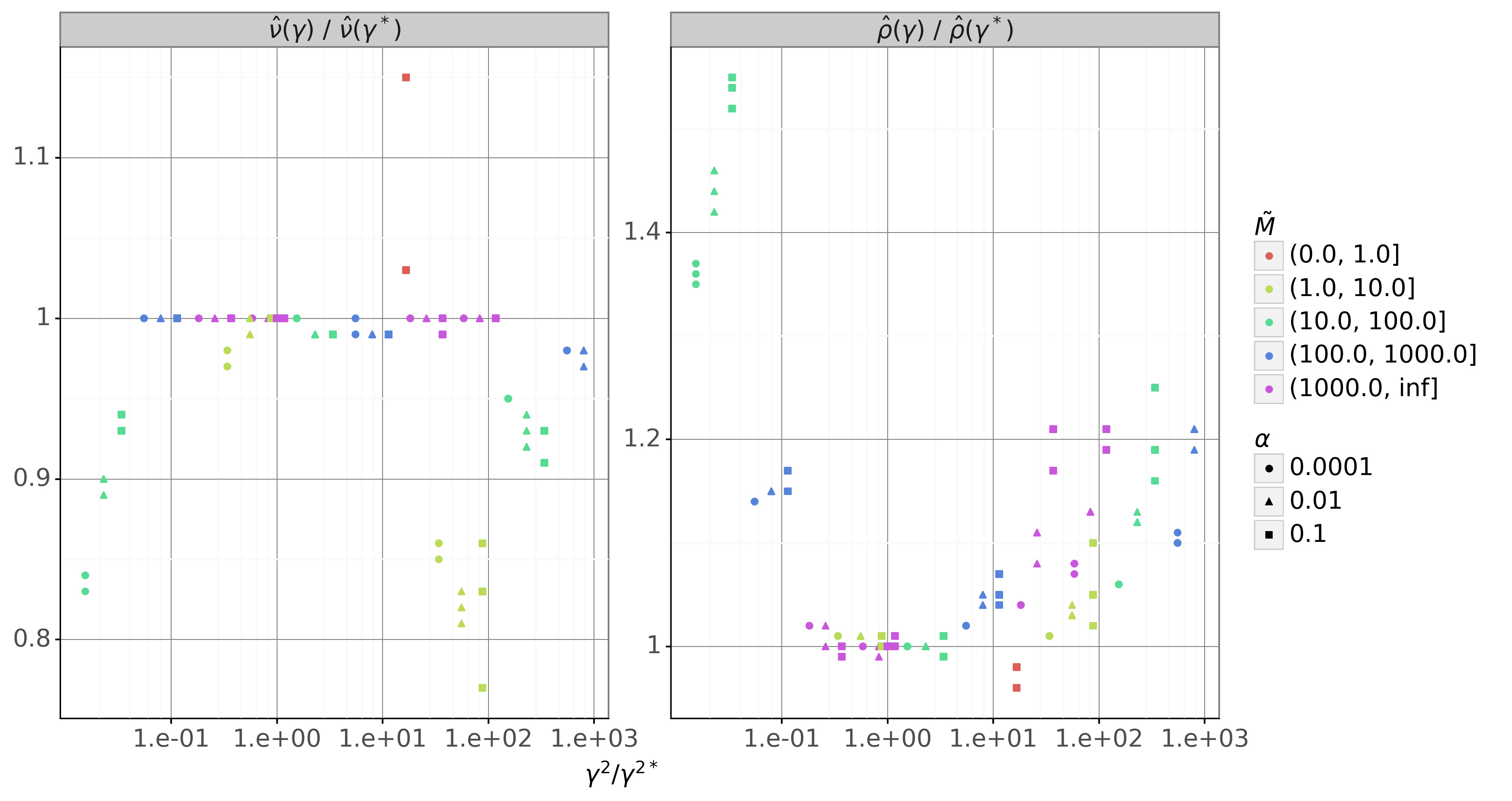}
\end{center}
\caption{\lp{Results of empirical investigation into the role of mixing distribution $H$ for the mSPRT. Four alternative mixing regimes are compared, $\gamma \in \{ 1, 1e^{-1} \tau, \tau, 1e^1 \tau \}$, to $\gamma^*$ (cf.~Theorem \ref{thm:optimal_mixture}), showing robustness in relative power ($\nu(\gamma) / \nu(\gamma^*)$) and run-length ($\rho(\gamma) / \rho(\gamma^*)$) to mixing misspecification of 1 order of magnitude across a variety of scenarios. The few cases where $\gamma \neq \gamma^*$ achieves superior results indicate a breakdown of asymptotics leading to Theorem \ref{thm:optimal_mixture}. Note: we do not show results for severely under-powered parameter combinations ($\hat{\nu} < 0.1$), as the variance in estimates distracts from the overall picture.}}
\label{fig:gamma_varying_sims}
\end{figure*}


\subsubsection{Empirical comparison to fixed-horizon testing}
\label{sec:fixedhorizon_emp}

Recall that Proposition \ref{prop:improvement} gives the improvement in expected run-length of the mSPRT decision rule over an optimized fixed-horizon test, asymptotically as $\alpha \to 0$.  We now evaluate this asymptotic result in a finite sample setting. Figure \ref{fig:msprt_to_fixed_runtime} shows the benefit from stopping early outweighs the cost of additional slack in always valid decision boundaries at reasonable power levels ($\hat{\rho} \geq 0.5$), regardless of parameter values.

We also make a similar comparison in Section \ref{sec:realworld}, using data from over 10,000 experiments with two streams (treatment and control) from the large-scale commercial A/B testing platform where these methods were deployed.  (See Figure \ref{fig:PropRunLength} and the discussion in Section \ref{sec:realworld} for further details.)

\begin{figure*}
\begin{center}
\includegraphics[width=0.5\textwidth]{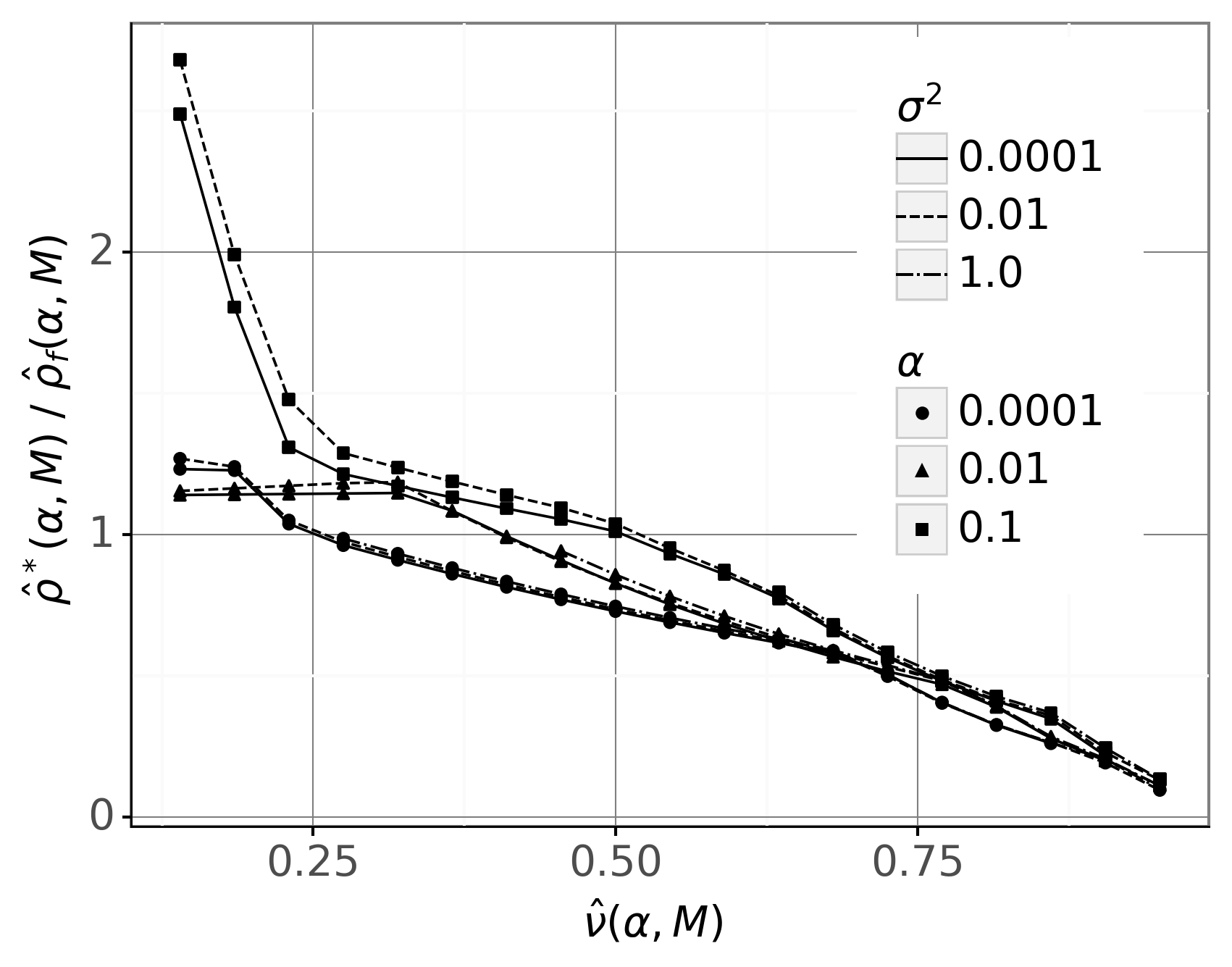}
\end{center}
\caption{\lp{Average run-time for optimally tuned mSPRT ($\rho^*$) and fixed-horizon UMP test ($\rho_f$), chosen to have same average power over distribution of effect sizes. The asymptotic result of Proposition \ref{prop:improvement} is shown to hold in many finite sample settings.}}
\label{fig:msprt_to_fixed_runtime}
\end{figure*}

\subsubsection{Empirical comparison to other sequential testing approaches}
\label{sec:comparison_emp}

As noted above, asymptotic first-order efficiency is almost certainly not unique to the mSPRT, even across $(M, \alpha)$ users. The discussion in Section \ref{sec:related_lil} of our paper, and in particular Section 4.1.1 of \cite{kaufmann2014complexity}, highlights tests of the form:

\[
T^{\beta}(\alpha) = \inf \left \{n \ : \ S_n > \left( \frac{2 \beta(n, \alpha)}{n} \right)^{1/2}  \right \}
\]

as a reasonable alternative class of candidates.  In this section we compare the mSPRT to two tests of the preceding form, from \cite{robbins1970statistical} and \cite{kaufmann2014complexity} respectively.  Formally, within the simulation framework specified above, we compare decision rules for $(M,\alpha)$ users derived from the following sequential tests:
\begin{enumerate}
\item the $H$-optimal mSPRT derived in this paper (denoted $\mathtt{mSPRT\_opt}$ in the plots);
\item the test proposed in Section 3 of \cite{robbins1970statistical}, characterized by $\beta(n, \alpha) = \frac{n+1}{n} \log( \frac{n+1}{2 \alpha})$ (denoted $\mathtt{r70}$ in the plots); and
\item a ``LIL based'' test with $\beta(n, \alpha) = \log(\alpha^{-1}) + 3 \log \log(\alpha^{-1}) + \frac{3}{2} \log \log(e * n)$ from \cite{kaufmann2014complexity} (denoted $\mathtt{k14}$ in the plots).
\end{enumerate}

In Figure \ref{fig:compare_pwr1_tests} we show that across all finite sample regimes examined, $\mathtt{mSPRT\_opt}$ has optimal average power and run-length while $\mathtt{r70}$ and $\mathtt{k14}$ vie for dominance. These results highlight a general phenomenon. While multiple decision rules may make the ``Goldilocks'' user perfectly efficient in the limit, they differ in preferential treatment for some $(M, \alpha)$ users over others. Concretely, the improved {\em rate} of efficiency gain for $\mathtt{k14}$ comes at the cost of lower efficiency for low to moderate $\tilde{M}$.

The mSPRT is given an advantage as it is optimally tuned to the user's $M$ via $\gamma^*$, but we stress that while some tuning is advised, finite sample efficiency is practically robust to $H$ misspecification. Even for users who truncate at 100x the typical run-length ($\tilde{M} = 100$), $\mathtt{mSPRT\_opt}$ has roughly 5\% improvement on average power over its peers, which is within the two order of magnitude misspecification range identified in Section \ref{sec:Hdependence_emp}. Average run-length remains 20\% better for all $\tilde{M} \geq 1e^{1}$, and up 40\% better in some cases ($\tau = 1e^{-2}, \alpha = 0.1$).

\begin{figure*}
\begin{center}
\includegraphics[width=0.9\textwidth]{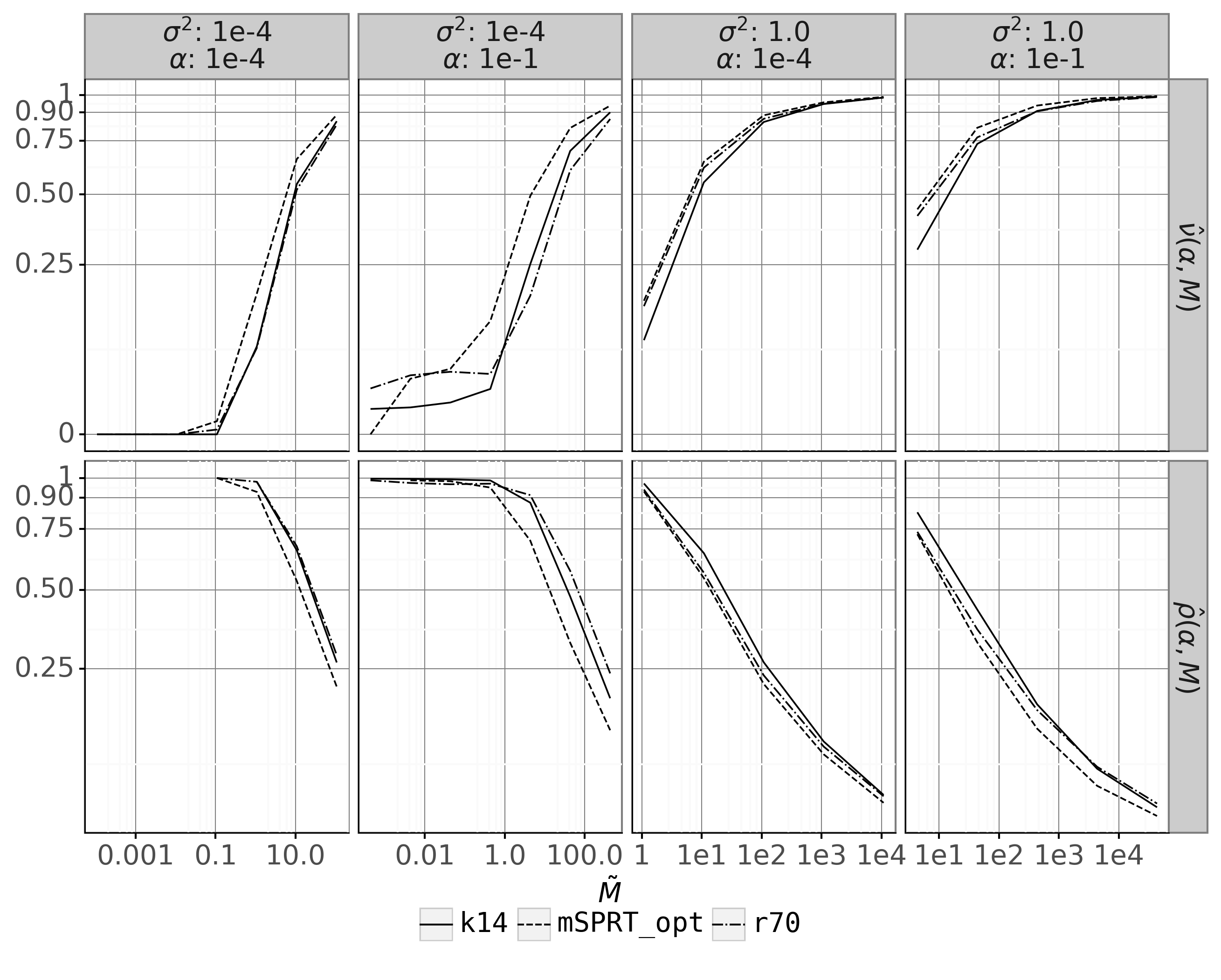}
\end{center}
\caption{\lp{Average power ($\nu$) and relative run-length ($\rho$) over a distribution of effect sizes, compared for 3 always valid decision boundaries: the $H$-optimal mSPRT ($\mathtt{mSPRT\_opt}$), a proposal in \cite{robbins1970statistical} ($\mathtt{r70}$), and a ``LIL based'' test from \cite{kaufmann2014complexity} ($\mathtt{k14}$). The optimally tuned mSPRT is strictly more efficient at all but the smallest levels of truncation, and continues to be substantially so even at very large truncation, despite the improved rate of efficiency gain for $\mathtt{k14}$.}}
\label{fig:compare_pwr1_tests}
\end{figure*}
}

\section{Deployment}
\label{sec:abtests}

One technical challenge remains before the always valid p-values and confidence intervals derived in the previous section may \djw{be} used in A/B testing. While those measures address inference for a single parameter in a single IID sequence of data, incoming visitors to an A/B test are actually randomized into two streams where each stream receives a distinct treatment.  In this section, we summarize how the measures can be modified to address the two most typical goals in practice: inference on the difference in means between two streams of normally distributed data, and inference on the difference in success probabilities for binary-valued data (both of these are commonly referred to as the difference in ``conversion rates").  The discussion is limited to p-values, but the results may be extended to confidence intervals in the usual way. Since January 2015, these two-stream p-values and confidence intervals have been implemented in a large scale platform serving thousands of clients, ranging from small businesses to large enterprises.

\djw{In the case of normal data, we develop a two-stream mSPRT which gives exact uniform Type I error control for testing the composite null hypothesis that the two means are equal. Extending the asymptotic theory of the previous section, we find that first-order efficiency for trading off power and run-time in the $\alpha \to 0$ limit is still obtained. As in the single stream case, our two-stream mSPRT is parameterized by a mixing parameter $H$ over the unknown treatment effect. Theorem \ref{thm:optimal_mixture} carries over to this setting, indicating how this mixture should be tailored to the distribution of anticipated effects, in order to obtain good performance at moderate $\alpha$. Much of the technical details are deferred to Appendix \ref{sec:supp_abtests}.

 For binary data, our two-stream mSPRT achieves approximate uniform Type I error control by appealing to the Central Limit Theorem. In this case, we use empirical data from our deployment to detail the efficiency gain over fixed-horizon testing. As in Sections \ref{sec:fixedhorizon} and \ref{sec:fixedhorizon_emp}, we see that the mSPRT, which is optimized to a prior for the treatment effect, can trade off power and run-time better than a comparably optimized fixed-horizon test. Of particular practical importance, the mSPRT is seen to outperform any fixed-horizon test that the experimenter might select herself, unless she has far better prior information than the platform does.}


\subsection{Two-stream p-values}

We represent observations in the two streams by IID sequences $\v{X} = (X_n)_{n =1}^\infty$ and $\v{Y} = (Y_n)_{n =1}^\infty$.  For normal data, we have $X_n \sim N(\mu_0, \sigma^2), Y_n \sim N(\mu_1, \sigma^2)$ where $\mu_0$ and $\mu_1$ are unknown, while the variance $\sigma^2$ is assumed known and common to both streams. For binary data, $X_n \sim Bernoulli(p_0), Y_n \sim Bernoulli(p_1)$, where $p_0, p_1 \in (0,1)$ are unknown. The conversion rate difference is $\theta = \mu_1 - \mu_0$ or $\theta = p_1 - p_0$ respectively.  In either case, we want to test the composite null hypothesis $H_0 : \theta = 0$ against $H_1: \theta \neq 0$.

We make the simplification that visitors arrive in pairs with one visitor assigned to each treatment, so that observations are obtained as a sequence of pairs $(W_n)_{n=1}^\infty = (X_1,Y_1),(X_2,Y_2),...$.  An always valid p-value is understood as a process adapted to the filtration generated by this sequence, which controls Type I error uniformly over both the composite null hypothesis and the choice of stopping time.  This model closely approximates the treatment allocation typically adopted in practice, where visitors arrive individually and each visitor is allocated to each treatment with 50\% probability independently of all other visitors.  A similar approach can be used even if the allocation to each group is not even, as long as it is fixed in advance.  Extensions to other allocation policies, such as data-dependent bandit schemes, will be the subject of future work.

For normal data, we view $(W_n)_{n=1}^\infty$ as a single stream of IID data from a bivariate distribution parameterized by the pair $(\theta, \mu)$, where $\mu = (\mu_0 + \mu_1)/2$.  It is straight forward to show that, after fixing $\mu = \mu^*$ arbitrarily, this distribution corresponds to the one-parameter exponential family $f_\theta(w) \propto \phi \left ( \frac{y - x - \theta}{\sigma \sqrt{2}} \right)$, where $w = (x,y)$.  Hence we may implement the mSPRT based on $f_\theta$, i.e. we threshold the mixture likelihood ratio given in (\ref{eq:msprt_definition}) with $s_n = \frac{1}{n} \sum_{i = 1}^n w_i$.  For any $\mu^*$, this mSPRT controls Type I error for testing the simple null $H_0 : \theta = 0, \mu = \mu^*$ against $H_1: \theta \neq 0, \mu = \mu^*$, and so the p-value derived from this mSPRT is always valid for testing the composite null hypothesis.  In Appendix \ref{sec:supp_abtests}, we show that it satisfies natural analogues of the single-stream optimality results described in Section \ref{sec:optimality}.

Unfortunately for binary data, the distribution of $W_n$ does not reduce to a one-parameter exponential family.  Nonetheless we set $p = (p_0 + p_1)/2$ and denote the density of $W_n$ by $f_{\theta,p}$.  Then, for any $\theta$ and $p^*$, in the limit as $n \to \infty$, the likelihood ratio against the pair $(\theta_0,p^*)$ in favor of $(\theta,p^*)$ approaches $(\tilde{f}_\theta(s_n)/\tilde{f}_{\theta_0}(s_n))^n$, where (with $w = (x,y)$):
$$ \tilde{f}_\theta(w) = \phi \left ( \frac{ y - x - \theta}{ \sqrt{p_0^*(1 - p_0^*) + p_1^*(1 - p_1^*)} } \right), \,\,\,\, p_0^* = p^* - \theta/2, \,\,\,\, p_1^* = p^* + \theta/2. $$
We compute the mSPRT p-values based on this density using the sample means in each stream as plug-in estimates for $p_0^*$ and $p_1^*$.  If $\alpha$ is moderate, the mSPRT terminates with high probability before this asymptotic distribution becomes accurate, so Type I error is not controlled.  However, for $\alpha$ small, simulation shows that these p-values are approximately always valid.

\subsection{Real-world improvement}
\label{sec:realworld}

\begin{figure*}
\begin{center}
\includegraphics[width=0.9\textwidth]{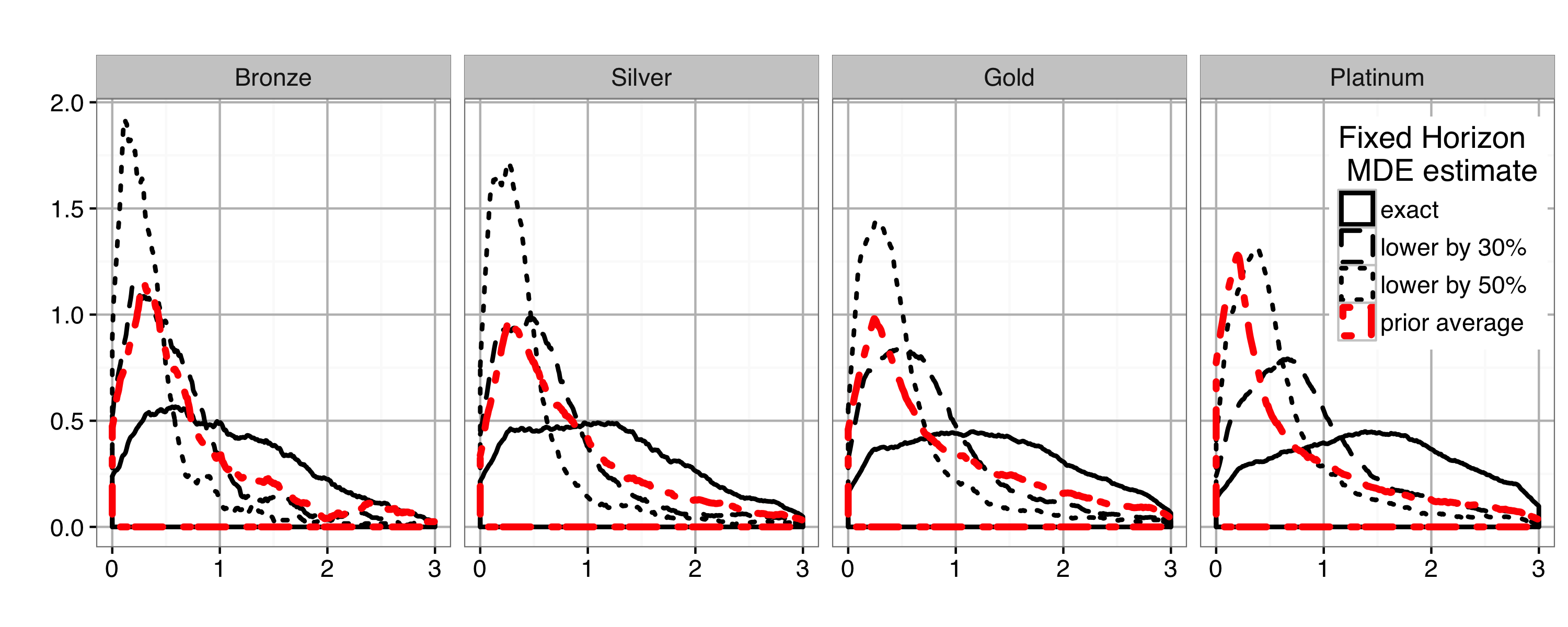}
\end{center}
\caption{The empirical distribution of sample size ratios between the mSPRT and \djw{suitably optimized fixed-horizon tests} over 10,000 randomly selected experiments, divided up by the subscription tier of the customer on the platform.  \rj{See main text for details.}}
\label{fig:PropRunLength}
\end{figure*}

\djw{Now we use empirical data to document the improvement of our two-stream mSPRT over fixed-horizon testing for binary data. For this purpose,} 10,000 experiments were randomly sampled from those binary experiments run on a large-scale commercial A/B testing platform in early 2015.

Customers of the platform in 2015 could purchase subscriptions at one of four tiers: Bronze, Silver, Gold or Platinum. \djw{Customers in the higher tiers tended to be larger, better optimized organizations, who were targeting smaller effect sizes. By separating out the 10 000 experiments according to the subscription tier of the experimenters, we can investigate how the two-stream mSPRT performs under different true effect distributions. This mirrors how we varied the prior $G$ for the numerical simulations presented in Section \ref{sec:optimality}. Specifically, for each tier, we found that the observed data was consistent with a normal distribution of true effect sizes $\theta/\sqrt{p_0^*(1 - p_0^*) + p_1^*(1 - p_1^*)}$ across experiments, so we fit a centered normal prior $G$ with unknown variance for the effect under the alternative hypothesis. For this fitting, shrinkage has been applied to the distribution of observed effect sizes via James-Stein estimation \citep{james1961estimation} to address the statistical noise in these observed values.

In Figure \ref{fig:PropRunLength}, we compare the run-time of the mSPRT optimized to the fitted distribution, $G$, against the fixed-horizon test at which 80\% average power over $G$ is obtained. The red curve is the empirical distribution for the ratio between the sample size where the mSPRT terminates and this fixed-horizon. For all tiers, the ratio falls below one with high probability. The black curves compare the mSPRT against the fixed-horizon test that the experimenter might choose if she has additional information about the effect size sought, beyond what is captured in the distribution, $G$. Here we suppose that she can estimate the unknown effect up to some specified relative error, and then she selects the sample size that provides 80\% power at her lower bound for the effect. In fact, a very precise estimate is required to achieve a run-time improvement over the mSPRT (a relative error below 50\% would rarely be achievable in practice). Further discussion of Figure \ref{fig:PropRunLength}, and of the broader practical gains associated with our two-stream mSPRT p-values, is given in our companion paper, \cite{johari2017kdd}.}

\section{Multiple testing}
\label{sec:multiple}

In this final section, we examine how always valid p-values and confidence intervals may be combined with existing multiple testing procedures when several experiments are conducted simultaneously. One option is to derive inference measures for each test individually, and these bound the expected proportion of the experiments that incur a Type I error, i.e., the false positive rate. However, that approach can be insufficient if the combination of Type I errors across multiple experiments can have a disproportionate impact on the user's ability to make good decisions. In the multiple testing literature, fixed-horizon p-values and confidence intervals are taken as input, and the procedures output {\it q-values} and corrected confidence intervals that are designed to satisfy a global error constraint that better reflects the overall cost to the user.

Obtaining the same error controls with a data-dependent sample size is highly non-trivial, even if the stopping rule is fixed by the platform. However, always validity provides an opportunity to do so, while still offering the user substantial latitude to choose her own stopping time. In what follows, we consider two leading multiple testing error constraints: family-wise error rate (FWER) and false discovery rate (FDR), defined below.  We obtain conditions on the user's stopping time that ensure these objectives can be bounded by supplying always valid p-values and confidence intervals as input to fixed-horizon procedures in popular use. In this sense, we say that these procedures, as well as the error constraints, {\em ``commute'' with always validity} over a class of stopping times. The resulting always valid {\it q-values} and corrected confidence intervals have both been adopted in the large-scale commercial A/B testing platform, in appropriate contexts.

We suppose that $m$ experiments are initiated at once, and at each successive step one observation is made simultaneously on every experiment.

\subsection{Error constraints}

We focus on the two error functions most extensively studied. The first is the {\em family-wise error rate} (FWER):
\begin{equation*}
\FWER  = \max_{\v{\theta}} \P_{\v{\theta}} (\delta_i = 1 \text{ for at least one } i \text{ s.t. } \theta^i = \theta_0^i ).
\end{equation*}
This is the worst-case probability of incurring any false positive.

The second is the {\em false discovery rate} (FDR):
\begin{equation*}
\FDR = \max_{\v{\theta}} \E_{\v{\theta}} \left \{ \frac{ \# \{ 1 \leq i \leq m : \theta^i = \theta_0^i, \delta_i = 1 \} }{ \# \{ 1 \leq i \leq m : \delta_i = 1 \} \wedge 1 } \right \}.
\end{equation*}
This is the worst-case average proportion of false positives among those experiments where the null hypothesis is rejected. As an example, consider a user who runs multiple experiments in order to compare the performance of the same two variations across different metrics. Each arriving visitor produces one observation for each experiment. FWER control across these experiments may be useful if she must prioritize performance on every metric, so a mistake in just one experiment can be very costly. FDR control may be useful if good performance on balance over many metrics is sufficient.

\subsection{Fixed-horizon procedures}

The goal of a multiple testing procedure is to make a decision on whether to reject or accept each null hypothesis, such that a global error constraint holds. In general, the existing multiple testing literature assumes a fixed-horizon framework. The standard procedure to control the FWER is the Bonferroni correction \citep{dunn1961multiple}: this takes fixed-horizon p-values as input and rejects hypotheses $(1),\ldots(j)$ where $j$ is maximal such that $p^{(j)} \leq \alpha/m$,  and $p^{(1)},\ldots,p^{(m)}$ are the p-values arranged in increasing order.  For FDR, the standard procedure is Benjamini-Hochberg \citep{benjamini1995controlling}, abbreviated as BH. Given fixed-horizon p-values, two versions of BH are used depending on whether the data are known to be independent across experiments.  If independence holds (BH-I), the procedure rejects hypotheses $(1),\ldots,(j)$ where $j$ is maximal such that $p^{(j)} \leq \alpha j / m$; in general (BH-G), the procedure chooses the maximal $j$ such that:
\begin{equation*}
p^{(j)} \leq \frac{\alpha j}{m \sum_{r = 1}^m 1/r}.
\end{equation*}

For the purposes of an A/B testing platform, such a procedure can be viewed as a mapping from the $m$ fixed-horizon p-values to so-called {\it q-values} that can be displayed on an identical dashboard (see Appendix). By thresholding each q-value at $\alpha$, a user can bound the given error function at her desired level. The practical advantages of p-values described in Section \ref{sec:introduction} are preserved: the same q-values can be used by many naive users, each with their own $\alpha$. 

Lastly, these procedures have similar interpretations for confidence intervals. Here the goal is to control false coverage when the user selects some subset of the experiments. One difference is while users typically only view the p-values of significant experiments, they may wish to gauge the range of plausible parameter values even in those experiments where the null hypothesis is not rejected. 

For the Bonferroni procedure, and any set of confidence intervals $I^i(\alpha)$, constructing new intervals $I^i(\alpha/m)$ bounds the probability that a confidence interval fails to cover the true value on any of the selected experiments, giving FWER control.

The analogy to FDR is the {\em False Coverage Rate (FCR)} \citep{benjamini2005false}: the expected proportion of the selected confidence intervals that incur false positives, set at zero if none are selected.  \cite{benjamini2005false} give a procedure to obtain FCR control at a fixed horizon when the experiment selection rule is known: the nominal level $\alpha$ is replaced by $R\alpha/m$ for some $R$ defined in terms of the rule. Here we extend their approach to address unknown selection rules in the fixed-horizon context, which we later use as a first step for sequential FCR control over classes of stopping times. 

We restrict to selection rules that are the union of the discoveries and some fixed set $J$ of experiments, with $j = |J| \ll m$, which are always of interest to the user.  Theorem \ref{thm:fixed_horizon_FCR} then gives a procedure which bounds the FCR in terms of $j$.  The proof is given in the Appendix.  For the procedure described in \cite{benjamini2005false}, it is the aggressive selection rules that choose few experiments which can obtain the highest FCR, and roughly speaking $R$ is a measure of how few experiments the rule can select.  Our approach is to be conservative over the unknown selection rule, taking $R$ for each interval to be the fewest number of experiments that could be selected, given that this interval corresponds to a selected experiment.

\begin{theorem}
\label{thm:fixed_horizon_FCR}
Given fixed-horizon p-values $\v{p}$, let $S^{BH}$ be the rejection set under BH-I, $R^{BH} = |S^{BH}|$, and $(CI^i(1-s))_{i =1}^m$ be the corresponding fixed-horizon CIs at each level $s \in (0,1)$.  Define the corrected confidence intervals:
\begin{equation}
\tilde{CI}^i =
\begin{cases}
CI^i(1-R^{BH}\alpha / m) & i \in S^{BH};\\
CI^i(1-(R^{BH}+1)\alpha / m) & i \notin S^{BH}.
\end{cases}
\end{equation}
Then for any $J$, if the selection rule is the experiments $J \cup S_{BH}$, the FCR is at most $\alpha(1 + j/m)$.
\end{theorem}

\subsection{Commuting with always validity}

Propositions \ref{prop:Bonferroni} and \ref{prop:bh_general} in Appendix \ref{sec:supp_multiple} establish that Bonferroni and BH-G commute with always validity on all p-value processes.  The reason is that, for any always valid p-values and any stopping time, the set of p-values evaluated at that time defines a set of fixed-horizon p-values. This is particularly useful as p-value processes may be replaced by q-value processes on a user's streaming dashboard and still enjoy always valid robustness guarantees. It is easy to show that Bonferroni commutes with always validity for confidence intervals as well.  

BH-I does not commute with always validity over independent p-value processes, however, because stopping times that depend on every experiment can introduce correlation in the p-values at that time (see the Appendix for an example).  Nonetheless, for many natural choices of this stopping time, FDR control is still achieved for any independent always valid p-values.  Theorem \ref{thm:sequential_fdr} gives a sufficient condition on the stopping time.

\begin{definition}
Given independent always valid p-values $\v{p}_n$, let $S^{BH}_n$ be the rejections when BH-I is applied to these at level $\alpha$ and let $R^{BH}_n = |S^{BH}_n|$. Define:
\begin{equation} \nonumber
\begin{split}
&T_r = \inf \{ t \ : \ R_t^{BH} = r \}; \\
&T_r^{+} = \inf \{ t \ : \ R_t^{BH} > r \}; \\
&T_r^i = \inf \{ t \ : \ p_t^i \leq \frac{\alpha r }{m} \}. \\
\end{split}
\end{equation}

Now, if $p_{(1),n}^{-i}, p_{(2),n}^{-i},\ldots$ are the p-values for the experiments other than $i$ placed in ascending order, consider a modified BH procedure that rejects hypotheses $(1),\ldots,(k)$ where $k$ is maximal such that $p_{(k),n}^{-i} \leq \alpha (k+1)/m$, in parallel to the fixed horizon approach in \cite{benjamini1995controlling}.
Define the rejection set $(S^{BH}_n)^{-i}_0$ as those obtained under the original BH-I procedure if $p_n^i = 0$.  Let $(R^{BH}_n)^{-i}_0 = |(S^{BH}_n)^{-i}_0|$ and define:
\begin{equation} \nonumber
\begin{split}
&(T_r)^{-i}_0 = \inf \{ t \ : \ (R^{BH}_n)^{-i}_0 = r \} \\
&(T_r^+)^{-i}_0 = \inf \{ t \ : \ (R^{BH}_n)^{-i}_0 > r \} .
\end{split}
\end{equation}
\end{definition}

We have the following theorem.  The proof can be found in Appendix \ref{sec:supp_multiple}.

\begin{theorem}
\label{thm:sequential_fdr}
Given a stopping time $T$, let $m_0$ be the number of truly null hypotheses and let $I$ be the set of null hypotheses $i$ such that:
\begin{equation}
\label{eq:seq_fdr_requirement}
\sum_{r=1}^m \P \left( (T_{r-1})^{-i}_0 \ \leq \ T < (T_{r-1}^+)^{-i}_0 \ \Big| \  T_r^i \leq T \ , \ T < \infty \right) > 1
\end{equation}
Then the rejection set $S^{BH}_T$ has FDR at most
$$\alpha \left(\frac{m_0}{m} +  \frac{|I| \sum_{k=2}^m \frac{1}{k}}{m} \right).$$
In particular, if we permit only stopping times where $I$ is empty, BH-I controls FDR and so commutes with always validity over all independent processes.
\end{theorem}

We can develop intuition for \eqref{eq:seq_fdr_requirement} by evaluating the condition on common examples. Perhaps the most natural stopping time for a user is the first time some fixed number $x \leq m$ hypotheses are rejected; i.e. $\ T = \inf_n \{ n \ : \ R_n = x \}$.  In that case,
\begin{gather*} 
\P \left( (T_{r-1})^{-i}_0 \ \leq \ T_x < (T_{r-1}^+)^{-i}_0 \ \Big| \  T_r^i \leq T_x \ , \ T_x < \infty \right)
= \P \left( (T_{r-1})^{-i}_0 \ \leq \ T_x < (T_{r-1}^+)^{-i}_0 \ \Big| \ T_x < \infty \right)
\end{gather*}
for each $i$.  This probability is 1 if $r = x$ and 0 otherwise, so $I$ is indeed empty and FDR is controlled.  On the other hand, a natural stopping time where $I$ is non-empty for some p-values is the first time that significance is reached in any of a given subset of experiments, where this subset has between two and $(m-1)$ elements.  A proof is given in Appendix \ref{sec:supp_multiple}, together with simulations showing that FDR control can be violated in this case.

Corrected confidence interval processes that give approximate FCR control can be derived with analogous restrictions by combining the results of Theorem \ref{thm:sequential_fdr} and the methods in Theorem \ref{thm:fixed_horizon_FCR}.
 
\begin{definition}
If $p_n^{i, \theta_0}$ is the p-value for testing $H_0: \theta^i = \theta_0$, let
\begin{align*}
&T_r^{i,\theta_0} = \inf \{ t \ : \ p_t^{i, \theta_0} \leq \frac{\alpha r }{m} \}\\
&(T_r)^{-i, J}_0 = \inf \{ t \ : \ |(S^{BH}_n)^{-i}_0 \cup J \backslash i| = r \} \\
&(T_r^+)^{-i, J}_0 = \inf \{ t \ : \ |(S^{BH}_n)^{-i}_0 \cup J \backslash i| > r \} .
\end{align*}
The last two stopping times denote the first times at least $r$ and more than $r$ experiments other than $i$, respectively, are selected.

If $p_{(1),n}^{-i}, p_{(2),n}^{-i},\ldots$ are the p-values for the experiments other than $i$ placed in ascending order, consider another modified BH procedure that rejects hypotheses $(1),\ldots,(k)$ where $k$ is maximal such that
$$ p_{(k),n}^{-i} \leq \alpha \frac{k}{m}, $$
These are the rejections obtained under the original BH-I procedure if $p_n^i = 1$.  We define stopping times associated with this procedure $(T_r)^{-i,J}_1$ and $(T_r^+)^{-i,J}_1$ analogous to the two stopping times above.
\end{definition}

We have the following theorem.

\begin{theorem}
\label{thm:sequential_FCR}
Given independent always valid p-values $\v{p}_n$ and corresponding CIs $(CI_n^i(1-s))_{i =1}^m$ at each level $s \in (0,1)$, define new confidence intervals:
\begin{equation}
\tilde{CI}_n^i =
\begin{cases}
CI_n^i(1-R_n^{BH}\alpha / m) & i \in S^{BH}_n;\\
CI_n^i(1-(R_n^{BH}+1)\alpha / m) & i \notin S^{BH}_n.
\end{cases}
\end{equation}
Let $J$ be a set of experiments and let $T$ be a stopping time such that the following conditions hold for every $i$, where $\theta^i$ is the true parameter value for that hypothesis:
\begin{align}
\sum_{r = 1}^m \P((T_r)^{-i, J}_0 \leq T < (T_r^+)^{-i, J}_0 | T_r^{i,\theta^i} \leq T < \infty) \leq 1;\\
\sum_{r = 1}^m \P((T_r)^{-i, J}_1 \leq T < (T_r^+)^{-i, J}_1 | T_r^{i,\theta^i} \leq T < \infty) \leq 1.
\end{align}
Then under the selection rule $J \cup S^{BH}_T$, the intervals $(\tilde{CI}_T^i)$ have FCR at most $\alpha(1 + j/m)$.
\end{theorem}

\section{Conclusion}
\label{sec:conclusion}

Our paper derives always valid p-values and confidence intervals for A/B testing.  These allow heterogeneous users to continuously monitor their experiments and to derive inferences at any time, which efficiently trade-off power and run-time for their needs.  In addition, we have identified how these measures may be combined with multiple hypothesis testing corrections to achieve sequential multiple testing controls.

We have only needed to place weak assumptions on the data generating processes; namely that observations are binary or normally distributed and independent across successive visitors entering the experiment.  However, in some contexts, A/B testing data can be heavily right-tailed \citep{fithian2014semiparametric}.  Various methods for modeling this skewed data are used in practice, and construction of always valid inference measures under these models would be a useful extension to this paper.  Further, the assumption of independence can fail due to seasonal effects, which induces correlations between visitors who arrive at similar times during the experiment.  Our implementation paper \citep{johari2017kdd} provides a simple heuristic (a ``reset policy") which identifies when seasonality may have led to inaccurate inference measures and applies conservative corrections to address this.  In future work, we aim to tackle the issue in more detail, deriving inference measures from an extension of the mSPRT that models the time dependence directly.

This paper has only considered the case where incoming visitors are randomized to treatments independently of the data collected so far.  In fact, it can be shown that the inference measures achieve Type I error control at any stopping time, even if treatments are assigned according to a bandit algorithm (see Section \ref{sec:related}).  Unifying these measures with bandit approaches would be valuable future work.  In particular, one might ask what allocation policy enables the mSPRT to achieve significance most quickly, or how the power and relative run-time profiles of always valid measures are impacted if the allocation is optimized to another objective such as regret minimization.

\theendnotes



\bibliographystyle{informs2014} 
\bibliography{sequential} 



\begin{APPENDICES}

\section{Proofs of optimality results}
\label{sec:supp_optimality}

\begin{proof}{Proof of Theorem \ref{thm:efficiency}.}
To establish asymptotic efficiency, given $(M,\alpha)$, it is sufficient to find some $\theta_1$, where for every feasible test $(T^*,\delta^*)$ with $\nu^*(\theta_1) \geq \nu(\theta_1;M,\alpha)$, we have that $\rho^*(\theta_1) \geq \rho(\theta_1;M,\alpha)(1 + o(1))$.

	Since the family $F_\theta$ can be equivalently viewed as exponential tilts of any $\theta' \in \Theta$, we assume $\theta_0 = 0$ wlog and write $I(\theta) := I(\theta,0)$.
\rj{Theorem 2 of \cite{LS77} can be used to establish that a normal approximation holds asymptotically for $\P_\theta(\delta(M, \alpha) = 0)$}; in particular, we have that for any fixed $\theta$,
\[\P_\theta ( \delta(M,\alpha) = 0 ) = \bar{\Phi} \left\{ \log (1 / \alpha)^{1/2} B(M,\alpha,\theta) \right\}(1 + o(1) )\]
where $B = \left( \frac{ I(\theta)^3 }{\theta^2 \psi^{''}(\theta)} \right)^{1/2} \left( \frac{ M }{ \log (1/\alpha) } - I(\theta)^{-1} \right).$  On the other hand, standard results on the log partition function $\psi$ imply that for fixed $(M,\alpha)$,
$$ \log (1/\alpha) ^{1/2} B(\theta) \sim \eta_2 \log (1/\alpha)^{1/2} \left(  \frac{ M \theta^2 }{ \log (1/\alpha) } - \eta_3 \right) $$
as $\theta \rightarrow 0$, where $\eta_2$ and $\eta_3$ are both positive constants.  Combining the two results, it follows that for $\theta_1 = \sqrt{\frac{\log (1/\alpha)}{M} ( \sqrt{2 / \eta_2} + \eta_3 )}$, we have that eventually
$$ \P_{\theta_1} ( \delta(M,\alpha) = 0 ) \leq \bar{\Phi}\left( \sqrt{2 \log (1/\alpha)} \right) := \beta_1 $$
i.e. the mSPRT has power at least $1 - \beta_1$ at $\theta_1$ in the limit.
Suppose that $(T^*,\delta^*)$ is another feasible test that achieves this power at $\theta_1$. Once $\alpha$ is sufficiently small that $0 < \alpha + \beta_1 < 1$, \rj{we can take advantage of a lower bound on the expected sample size of sequential testing procedures \citep{hoeffding1960lower} to show that for any $\theta \in (0, \theta_1)$,}
\begin{align*}
\E_\theta(T^*) &\geq \frac{ | \log( \alpha + \beta_1) | - \frac{1}{2} \theta_1^2 \psi''(\theta)  | \log( \alpha + \beta_1) |^{1/2} }{\max\{ I(\theta), I(\theta,\theta_1) \}}\\
&= I(\theta)^{-1} \log (1/\alpha) (1 + o(1) )
\end{align*}
By continuity, the result holds at $\theta_1$ also.  Comparing the above expression with (\ref{eq:limit_prob}) gives the desired inequality on the relative run-times at $\theta_1$.\hfill$\Box$
\end{proof}

\begin{proof}{Proof of Proposition \ref{prop:conservative}.}
Again wlog we assume $\theta_0 = 0$.  We fix $\theta \neq 0$, and for contradiction, we suppose that there is some $\beta < 1$ such that feasible tests with $\P_{\theta}(\delta^*(M,\alpha) = 0) \leq \beta$ exist in this limit.  Combining this Type II error bound with the Type I error bound at $\alpha$, the same lower bound of \cite{hoeffding1960lower} used above implies the existence of some $\kappa$ such that
\[ \E_\theta(T^*) \geq \kappa \log (1/\alpha) (1 + o(1)). \]
In the limit, this expectation exceeds $M$, so $T^*$ must certainly exceed $M$ with positive probability.\hfill$\Box$
\end{proof}

We now prove three lemmas that will let us prove Theorem \ref{thm:optimal_mixture} and Proposition \ref{prop:improvement}.
\begin{lemma}
\label{lem:exp_cgs}
Given $H, \theta \neq \theta_0$, there exists a $\lambda > 0$ such that for any $0 < \epsilon < 1$
\begin{equation} \P_\theta \left \{ | T^H(\alpha) - \frac{\log (1/\alpha)}{I(\theta, \theta_0)} | > \epsilon \, \left ( \frac{\log (1/\alpha)}{I(\theta, \theta_0)} \right ) \right \} = O(\alpha^{\lambda}). \end{equation}
\end{lemma}
\begin{proof}{Proof.}
	\rj{The proof follows by combining Lemmas 2 and 3 of \cite{PS1975} with Lemma 6 of \cite{LW94}.} \lp{Lemma 3 of \cite{PS1975} provides the upper bound for the case $T^H(\alpha) < \frac{\log (1/\alpha)}{I(\theta, \theta_0)}$. The other comes from well known exponential concentration bounds on the maximum deviation of a sample average from the corresponding mean (Lemma 2 of \cite{PS1975}), combined with the method of proof of Lemma 6 of \cite{LW94}, which shows similar bounds for stopping times of the form

\[ T_a = \inf \{ n \geq n_a : n \xi (S_n)  \geq a \}, \]

where $\xi$ is a smooth, positive function and $S_n$ is the sample average after $n$ observations, and a standard application of Jensen's inequality to bound $T^H(\alpha)$ by stopping times of the above form.}\hfill$\Box$
\end{proof}

\begin{lemma}
\label{lemma:ess_truncated}
Let
$$ A = \left \{ \theta : I(\theta, \theta_0) \geq \frac{\log (1/\alpha)}{M} \right \}.$$
Then there holds:
\begin{equation} \label{eq:thm1a}
M \rho(M,\alpha) =  \E_{\theta \sim G} \left\{ \v{1}_A \E_\theta ( T^H(\alpha) ) \right\}  + M \, Pr_{G(\theta)} \left( \bar{A} \right) + o(1).
\end{equation}
\end{lemma}
\begin{proof}{Proof.}
Let $0 < \epsilon < 1$. Define two times, $n_1 = (1-\epsilon) \log(1/\alpha) / I(\theta, \theta_0)$, $n_2 = (1+\epsilon) \log(1/\alpha) / I(\theta, \theta_0)$.
For $\theta \in \bar{A}$, we have the following bounds, where the final inequality is an application of (\ref{eq:limit_prob}) in probability:
\begin{align*}
M \geq \E_\theta(T(M,\alpha)) &\geq (n_1 \wedge M) \P_\theta( T^H(\alpha) > n_1)\\
&\geq (1 - \epsilon) \, M \, \P_\theta( T^H(\alpha) > n_1) \geq (1 - \epsilon) \, M + o(1).
\end{align*}

Let
$$ B^\epsilon = \left \{ \theta : I(\theta, \theta_0) \geq (1 + \epsilon) \left ( \frac{\log (1/\alpha)}{M} \right ) \right \},$$
for $\theta \in B^\epsilon$, $M \geq n_2$. Thus
$$ \E_\theta(T^H(\alpha)) \geq \E_\theta( T(M,\alpha) ) \geq \E_\theta( T(n_2,\alpha) ) \geq \int_{T^H(\alpha) \leq n_2} T^H(\alpha) d\P_\theta = \E_\theta(T^H(\alpha)) - \int_{T^H(\alpha) > n_2} T^H(\alpha) d\P_\theta. $$
By Cauchy-Schwartz, (\ref{eq:limit_prob}) in $\mathcal{L}^2$ and Lemma \ref{lem:exp_cgs},
$$ \int_{T^H(\alpha) > n_2} T^H(\alpha) d\P_\theta \leq \left( \E_\theta(T^H(\alpha)^2) \P_\theta(T^H(\alpha) \geq n_2) \right)^{1/2} = O( \alpha^{-\lambda/2} \log \alpha^{-1/2}) = o(1). $$

For $\theta \in A \backslash B^\epsilon$,
\begin{align*}
\E_\theta(T^H(\alpha)) \geq \E_\theta( T(M,\alpha) ) &= \E_\theta(T^H(\alpha)) + \int_{T^H(\alpha) \geq n_2} \{ n_S - T^H(\alpha) \} \, d\P_\theta\\
&\geq \E_\theta(T^H(\alpha)) + \int_{M \leq T < n_2} \{ n_S - T^H(\alpha) \} d\P_\theta - \int_{T^H(\alpha) > n_2} T^H(\alpha) d\P_\theta\\
&\geq \E_\theta(T^H(\alpha)) - (n_2 - M) + o(1) \geq \E_\theta(T^H(\alpha)) - \epsilon \, M + o(1).
\end{align*}

Putting the three cases together, we integrate over $\theta \sim G$ to obtain (\ref{eq:thm1a}), up to some error linear in $\epsilon$. To justify this step, it is easy to check that each term is finite. The result now holds on letting $\epsilon \to 0$.\hfill$\Box$
\end{proof}

\begin{lemma} \label{lemma:fast_detection_asymptotic_power}
	Let
	\begin{equation} \nonumber \begin{split}
	C(\alpha) &= \int_0^1 \Phi \left( \sqrt{ \frac{1}{2} \log (1/\alpha) } ( x^2 - 1) \right) dx, \\
	C_{f}(\alpha) &= \int_0^1 \Phi \left( \sqrt{ \log (1/\alpha) } (x - 1 ) \right) dx.
	\end{split}
	\end{equation}
	For the prior $G = N(0,\tau^2)$, let $\nu(M,\alpha)$ be the average power of the mSPRT. If $M = O(\log(1/\alpha))$,
	\[ \nu(M,\alpha) \sim C(\alpha) \frac{2 \sqrt{2}}{\tau} \left( \frac{ \log(1/\alpha)}{M} \right)^{1/2} \]
	Let $\nu_f(n,\alpha)$ be the average power of the fixed-horizon test with sample size $n$.  If $n = \Omega(\log(1/\alpha))$,
	\[ \nu_f(n,\alpha) \sim C_f(\alpha) \frac{2 \sqrt{2}}{\tau} \left( \frac{ \log (1/\alpha)}{n} \right)^{1/2} \]
\end{lemma}
\begin{proof}{Proof.}
	Wlog we suppose $\theta_0 = 0$. We begin with the fixed horizon result. It is simple to show that $z_{1 - 2 \alpha} \sim \sqrt{2 \log(1/\alpha) }$ as $\alpha \rightarrow 0$. Hence
	$$ \nu_f(\theta) = \bar{\Phi}\left( | \theta| \sqrt{n} - z_{1 - 2\alpha} \right) = \bar{\Phi}\left( \log(1/\alpha)^{1/2} S_f(\theta,n,\alpha)  \right) $$
	where $S_f(\theta,n,\alpha) = |\theta| \left( \frac{ \log (1/\alpha)}{n} \right)^{-1/2} - \sqrt{2} $.

	Let $B_f = \{ \theta \ : \ A_f(\theta,n,\alpha) \geq \sqrt{2} \}$. We split up the average power as
	$$ \nu_f = \int_{B_f} \nu_f(\theta) \frac{1}{\tau} \phi(\theta / \tau) d\theta + \int_{\bar{B}_f} - \nu_f(\theta) \frac{1}{\tau} \phi(\theta / \tau)d\theta $$
	denoting the two terms by (i) and (ii) respectively. For $\theta \in B_f$, the standard tail bound on the normal CDF, $\bar{\Phi}(x) \leq x^{-1} \phi(x)$ gives
	\begin{equation} \nonumber
	\begin{split}  \bar{\Phi}\left( \log (1/\alpha)^{1/2} S_f(\theta,n,\alpha)  \right) &\leq ( 4 \pi \log(1/\alpha) )^{-1/2} \alpha \\ &= o(\alpha),
	\end{split}
	\end{equation}
	so that $(i) = o(\alpha)$ as well. For term $(ii)$, we note that $\bar{B}_f \rightarrow \{0\}$ so that $\phi(\theta / \tau) \sim 1$. This, the change of variable $x = \left( \frac{ 2 \log(1/\alpha)}{n} \right)^{-1/2} \theta$ and symmetry of the integrand give
	\[ (ii) \sim \frac{2 \sqrt{2}}{\tau} \left( \frac{ \log(1/\alpha)}{n} \right)^{1/2} \int_0^2 \bar{\Phi} \left( ( \log \alpha^{-1} )^{1/2} (x - 1) \right) dx. \]
	The result follows on noting $ \bar{\Phi} \left( \log(1/\alpha)^{1/2} (x - 1) \right) = o(1)$ when $x > 1$.

	For the mSPRT, we use the normal approximation to the tail probabilities of the mSPRT stopping time from Theorem 2 of \cite{LS77} which, in the case of standard normal data, gives
	\[ \P_\theta ( T(\alpha) > M ) \sim \bar{\Phi} \left\{ \log (1/\alpha)^{1/2} S(\theta,M,\alpha) \right\} \]
	where $S(\theta,M,\alpha) = \frac{1}{2 \sqrt{2}} \left \{  \theta^2 \left (\frac{ \log (1/\alpha)}{M} \right )  - 2 \right \} $. The rest of the proof proceeds as for the fixed horizon test, except with $B = \{ \theta \ : \ S \geq \sqrt{2} \}$ and changes in the integrand of $(ii)$ as stated in the proposition.\hfill$\Box$
\end{proof}

\lp{
The next two results rely on equation (67) of \cite{LS1979}, which gives an approximation to $\E_\theta ( T(\alpha) )$ as $\alpha \rightarrow 0$. We re-print the equation here in our notation for easy reference:
\begin{equation}
\label{eq:ls1979_67}
I(\theta, \theta_0) \E_\theta ( T(\alpha) ) = \log \alpha^{-1} + \frac{1}{2}\left[ \log\left(\frac{\log \alpha^{-1}}{I(\theta, \theta_0)}\right) - \log\left( \frac{2\pi (H'(\theta))^2}{\psi''(\theta)}\right) - \frac{\sigma^2}{\psi''(\theta)}\right] + \frac{\E[R^2]}{2 \E[R]} + o(1)
\end{equation}
where prime and double prime denote first and second derivative, $H'$ is assumed to exist in a neighborhood of $\theta$, $\sigma^2 = \E X_1^2 - (\E X_1 )^2$, and $R = \inf \{ n : S_n > 0 \}$, the renewal time of $S_n$ at 0.

The authors also note that while the renewal term is in general difficult to evaluate, it simplifies in the normal case  $f_\theta(x) = \phi(x-\theta)$ to the following expression:
\[  \frac{\E[R^2]}{2 \E[R]} = 2 + \frac{1}{2} \theta^2 - 2 \theta B \left( \frac{\theta}{2} \right) \]
where
\[ B(u) = \sum_{k = 1}^{\infty} k^{-1/2} \phi( u k^{1/2} ) - u \Phi ( - u k^{1/2} ), \]
which we use in the proof of Proposition \ref{prop:improvement} below.
}

\begin{proof}{Proof of Theorem \ref{thm:optimal_mixture}.}

	Combining Lemma \ref{lemma:ess_truncated} with \eqref{eq:ls1979_67}, we find that, up to $o(1)$,
	$$ M \rho(M,\alpha) = -2 \E_{\theta \sim G} \v{1}_A I(\theta, \theta_0)^{-1} \log h_\gamma(\theta) + K(G, \alpha)$$
	for some function $K$ not depending on $\gamma$. The stated $\gamma^*$ is the minimizer of this expression.
\hfill$\Box$
\end{proof}

\begin{proof}{Proof of Proposition \ref{prop:improvement}.}
From Lemma \ref{lemma:fast_detection_asymptotic_power}, we see that to match the average power of the truncated mSPRT ($\nu_f = \nu$), the calibrated fixed-horizon test must have sample size $O(M)$; i.e. $\rho_f = O(1)$.  Thus it is sufficient to show that $\rho = o(1)$.

Again we take $\theta_0 = 0$ wlog, so for standard normal data $I(\theta, 0) = \theta^2 / 2$. We invoke Lemma \ref{lemma:ess_truncated} and attack the two terms in that result separately. Let $\delta = \left( \frac{2 \log (1/\alpha)}{M} \right)^{1/2}$.
\[ Pr_{\theta \sim N(0,\tau)} ( \bar{A} ) = 2 \int_0^\delta \frac{1}{\tau} \phi(\theta / \tau ) d \theta \sim \frac{2 \sqrt{2}}{\tau} \left( \frac{ \log (1/\alpha)}{M} \right)^{1/2} = o(1) \]
by similar arguments to those in the proof of Lemma \ref{lemma:fast_detection_asymptotic_power}.

The first term in Lemma \ref{lemma:ess_truncated} is more complicated. By equation \eqref{eq:ls1979_67} for the case of Normal data,
\begin{equation} \label{exp_time_decomp}
\begin{split}  \E_\theta ( T(\alpha) ) &= 2 \theta^{-2} \log (1/\alpha) + \theta^{-2} \log \log (1/\alpha) + D_1 \theta^{-2} \log | \theta | \\
&+ D_2 \theta^{-2} + D_3 + D_4 \theta^{-1} B( \theta / 2) + o(1)
\end{split}
\end{equation}
as $\alpha \rightarrow 0$.
It remains to show that each term in (\ref{exp_time_decomp}) has $o(M)$ expectation on $1_A$. We focus on terms 1, 3 \& 6 as the remainder are clearly lower order.

{\em Term 1.}
$$ \E_{\theta \sim N(0,\tau)} ( \v{1}_A \theta^{-2} ) = 4 \int_\delta^\infty \theta^{-2} \frac{1}{\tau} \phi \left ( \frac{\theta}{ \tau} \right) d \theta = \frac{4}{\tau^2} \left \{ \frac{\tau}{\delta} \phi \left( \frac{\delta}{\tau} \right) - \bar{\Phi} \left( \frac{\delta}{\tau} \right) \right \} $$
which is bounded over $\alpha \in (0, 1 - \epsilon)$. It follows that
\[  \E_{\theta \sim N(0,\tau)} ( \v{1}_A \theta^{-2} ) \log (1/\alpha) \sim \frac{ 2 \sqrt{2} }{\tau} M \left( \frac{ \log (1/\alpha)}{M} \right)^{1/2} = o(M). \]

{\em Term 3.} By calculus,
\begin{equation} \nonumber \begin{split}
 &\E_{\theta \sim N(0,\tau)} ( \v{1}_A \theta^{-2} \log | \theta | \} \propto \int_\delta^\infty \theta^{-2} \log \theta e^{-\theta^2 / 2 \tau^2 } \\
 &= \frac{1}{4} \left[ \frac{1}{\sqrt{2} \tau} \Gamma \left( - \frac{1}{2} \ , \ \frac{\delta^2}{2 \tau^2} \right) \log \delta \right. \\
 &+ \left. \delta^{-1} \text{MeijerG} \left( \{ \{\}, \{ \frac{3}{2} , \frac{3}{2} \} , \{ \{0,\frac{1}{2},\frac{1}{2}\},\{\}\} , \frac{\delta^2}{2 \tau^2} \right) \right],
 \end{split} \end{equation}
The $\text{MeijerG}$ term is asymptotically constant as $\delta \rightarrow 0$, and
\[ \Gamma \left( - \frac{1}{2} \ , \ \frac{\delta^2}{2 \tau^2} \right) \rightarrow 2 \sqrt{2} \tau \delta^{-1}. \]
It follows that
$$ \E_{\theta \sim N(0,\tau)} ( \v{1}_A \theta^{-2} \log | \theta | ) \propto M \left[ K_1 \left( M \log (1/\alpha) \right)^{-1/2} + K_2 \left( \frac{ \delta \log \delta }{\log (1/\alpha)} \right) \right] $$
where $K_i$ are both constants depending on $\tau$. Both terms in the bracketed sum clearly converge to $0$ since $\delta \rightarrow 0$.

{\em Term 6.} By standard bounds on the normal CDF, $B(u) \geq 0$ and
\begin{equation} \nonumber \begin{split} &B(u) \leq u^{-2} \left( \int_1^\infty x^{-3/2} \phi(u x^{1/2}) dx + \phi(u) \right) \\
&= \theta^{-2} \left( 3 \phi(\theta) - 2 \theta \Phi( - \theta ) \right)
  \end{split} \end{equation}
Hence,
\begin{equation} \nonumber \begin{split}  \E_{\theta \sim N(0,\tau)} ( \v{1}_A \theta^{-1} B( \theta / 2)) \leq K_3 \, \E_{\theta \sim N(0,\tau)} ( \v{1}_A \theta^{-3} \phi(\theta / 2) ) + K_4 \\
\leq K_4 \delta^{-2} e^{K_5 \delta^2} - K_7 \Gamma(0, K_6 \delta^2 ) + K_4,
\end{split}\end{equation}
where $\delta^{-2} = M / \log (\alpha^{-2}) = o(M)$ and $\Gamma(0, K_6 \delta^2 ) \sim \log K_6 / \delta^2 = O( \log \delta ) = o(M)$.
\hfill$\Box$
\end{proof}

\section{Optimality for two-stream normal data}
\label{sec:supp_abtests}

Given a choice of mixing distribution $H$, the two-stream p-values are derived from the mSPRT which rejects the null if
\begin{equation} \Lambda_n^H(S_n) = \int_{\Theta} \left( \frac{\tilde{f}_\theta(S_n)}{\tilde{f}_{\theta_0}(S_n)}\right)^n dH(\theta)
\end{equation}
ever exceeds $\alpha$, where $S_n = \frac{1}{n} \sum_{i=1}^n W_i$.  First we notice that $\Lambda_n^H$ depends on the data only through $(-1,1)^T S_n \sim N(\theta, 2\sigma^2/n)$, so the power and the run-time of this test do not depend on $\mu$.  Let $\nu^H(\theta; M, \alpha)$, $\rho^H(\theta; M, \alpha)$ be the power and average relative run-length of the truncated test.  We say that the {\em relative efficiency} of this test at $(M, \alpha)$ is
\[ \phi^H(M, \alpha) = \inf_{(T^*, \delta^*)} \inf_{\theta \neq \theta_0, \mu} \frac{\rho^*(\theta, \mu)}{\rho^H(\theta; M, \alpha)} \]
where the infimum is taken over all tests with $T^* \leq M$ a.s., $\sup_\mu \nu(\theta_0, \mu) \leq \alpha$, and for all $\theta \neq \theta_0$, $\inf_\mu \nu^*(\theta,\mu) \geq \nu^H(\theta; M, \alpha)$.

\begin{proposition}
\label{prop:efficiency2}
	For any $H$, if $\alpha \to 0, M \to \infty$ such that $M = O(\log(\alpha^{-1}))$, we have $\phi^H(M, \alpha) \to 1$.
\end{proposition}
\begin{proof}{Proof.}
	Fix $\mu = \mu^*$ arbitrarily.  Then any $(T^*,\delta^*)$ satisfying the above conditions is also feasible for testing $H_0 : \theta = \theta_0, \mu = \mu^*$ against $H_1 : \theta \neq \theta_0, \mu = \mu^*$.  The result follows by Theorem \ref{thm:efficiency}.\hfill$\Box$
\end{proof}

Now we consider any prior for the pair $(\theta,\mu)$ under $H_1$, such that $\theta \sim N(0,\tau^2)$ marginally.  For normal mixtures $H = N(0,\gamma^2)$, let $\rho_\gamma(M,\alpha)$ be the average power and relative run-time over this prior.
\begin{proposition}
\label{prop:optimal_mixture2}
	To leading order as $\alpha \to 0, M \to \infty, M = O(\log(\alpha^{-1}))$, $\rho_\gamma$ is minimized by
	$$ \gamma^{2*} = \tau^2 \frac{ \Phi(-b) }{ \frac{1}{b} \phi(b) - \Phi(-b)} $$
	where $b = \left( \frac{2 \sigma^2 \log \alpha^{-1} }{M \tau^2} \right)^{1/2}$.
\end{proposition}
\begin{proof}{Proof.}
	Immediate from Theorem \ref{thm:optimal_mixture}.
\hfill$\Box$
\end{proof}

Now we compare the truncated mSPRT to the fixed-horizon t-test based on the difference between the sample means in the two streams, which is calibrated to have the same average power on this prior.  Noting that the fixed-horizon sample size does not depend on $\mu$, we see that Proposition \ref{prop:improvement} carries over to two-stream normal data.

\section{Multiple testing}
\label{sec:supp_multiple}

\subsection{Commutativity with always valid p-values}

\begin{proposition}
\label{prop:Bonferroni}
Let $(p_n^i)_{i = 1}^m$ be always valid p-values, and let $T$ be an arbitrary stopping time.
Then the set of decisions obtained by applying Bonferroni to $\v{p}_T$ controls FWER at level $\alpha$.
\end{proposition}
\begin{proof}{Proof.}
For all $\v{\theta}$, the variables $p_T^1,...,p_T^m$ satisfy the property that, for the truly null hypotheses $i$ with $\theta^i_0 = \theta^i$, $p_T^i$ is marginally super-uniform.  Hence there is a vector of (correlated) fixed-horizon p-values with the same distribution as $\v{p}_T$, and so Bonferroni applied to the always valid p-values must control FWER.
\hfill$\Box$
\end{proof}

\begin{proposition}
\label{prop:bh_general}
Let $(p_n^i)_{i = 1}^m$ be always valid p-values, and let $T$ be an arbitrary stopping time.
The set of decisions obtained by applying the BH-G procedure to $\v{p}_T$ controls FDR at level $\alpha$.
\end{proposition}

We cannot use the same proof here as for Proposition \ref{prop:Bonferroni}, because Theorem 1.3 in \cite{benjamini2001control}, which establishes that BH-G controls FDR under arbitrary correlation, requires that the fixed-horizon p-values be strictly uniform (rather than super-uniform). Instead, we rely on the following lemma.

\begin{lemma}
\label{lemma:bhq_log_bound}
\[ \sup_{f \in \cF} \sum_{k=1}^m \frac{1}{k} \int_{(k-1) \alpha / m}^{k\alpha / m} f(x) dx = \frac{\alpha}{m} \sum_{k=1}^m \frac{1}{k} \]
where $ \cF = \{ f:[0,1] \rightarrow \reals_{+} \ : \ F(x) = \int_0^x f(x) dx \leq x \ , \ F(1) = 1 \} $, $m \geq 1$, and $0 \leq \alpha \leq 1$.
\end{lemma}
\begin{proof}{Proof.}
Since $f \in \cF$ are bounded, we restate the optimization in terms of $F_k = F(\frac{k\alpha}{m})$, and $F_0 \equiv 0$,
\begin{equation} \begin{split} \nonumber
\sup_{F_1 , \dots , F_m } &\sum_{k=1}^m \frac{1}{k} ( F_k - F_{k-1} ) \\
\text{subject to} & \,\,\,\, 0 \leq F_j \leq \frac{k\alpha}{m} \ , \ F_k \geq F_{k-1} \ k = 1 ,\dots, m.
\end{split} \end{equation}
The objective can be rearranged as
\[ \sum_{k=1}^{m-1} \frac{1}{k(k+1)} F_k + \frac{1}{m} F_m \]
which is clearly maximized by $F_k = \frac{k\alpha}{m}$ for all $k$.
\hfill$\Box$
\end{proof}

\begin{proof}{Proof of Proposition \ref{prop:bh_general}.}
Adapting the proof given in \cite{benjamini2001control} now is straight-forward.  Translating into the sequential notation of this paper, the only non-immediate step is to show
\begin{equation} \begin{split} \nonumber
 &\sum_{k=1}^m \frac{1}{k} \sum_{r=k}^m \P \left( T_k^i \leq T < T_{k-1}^i \ , \ T_r \leq T < T_r^{+} \ , \ T \leq \infty \right) \\
 \leq &\sum_{k=1}^m \frac{1}{k} \P \left( \frac{(k-1)\alpha}{m} \leq p_T^i \leq \frac{k\alpha}{m} \right) \leq \frac{\alpha}{m} \sum_{k=1}^m \frac{1}{k} 
 \end{split}
\end{equation}
for all truly null hypotheses $i$. The first inequality is a restatement of definitions, and the second follows from Lemma \ref{lemma:bhq_log_bound} since by always-validity $p_T^i$ is super-uniform.
\hfill$\Box$
\end{proof}

\begin{proof}{Proof of Theorem \ref{thm:sequential_fdr}.}
We assume wlog that the truly null hypotheses are $i = 1,\ldots,m_0$.  Letting $V_n$ denote the number of true null rejected at $n$, the FDR can be expanded as
\begin{equation} \nonumber
\E \left( \sum_{r=1}^m \ \frac{1}{r} \ V_T \ 1_{ \{ T_r \ \leq \ T < T_r^{+} \} }  1_{T < \infty}  \right) = \sum_{i=1}^{m_0} \sum_{r=1}^m \frac{1}{r} \P \left( T_r^i \leq T \ , \ T_r \ \leq \ T < T_r^{+} \ , \ T < \infty \right).
\end{equation}
Note that the sets $ \{ T_r \ \leq \ T < T_r^{+} \} $ are disjoint and cover any location of $T$. Consider the terms in the sum over $i \in I$ and $i \notin I$ separately. For $i \notin I$, we bound the probability in the third equality by
\begin{equation} \nonumber
\begin{split}
\P \left( T_r^i \leq T \ , \ T < \infty \right)& \P \left(  T_r  \leq T < T_r^{+} \ \Big| \  T_r^i \leq T \ , \ T < \infty \right) 
 \leq \frac{\alpha r}{M}  \  \P \left(  T_r \leq T < T_r^{+} \ \Big| \  T_r^i \leq T \ , \ T < \infty \right) \\
 &= \frac{\alpha r}{M}  \  \P \left( (T_{r-1})^{-i}_0 \ \leq \ T < (T_{r-1})^{-i+}_0 \ \Big| \  T_r^i \leq T \ , \ T < \infty \right)
\end{split}
\end{equation}
where the first inequality follows from always-validity of sequential p-values, and the last equality because the modified BH procedure on the $m-1$ hypothesis other than the $i$th makes equivalent rejections at time $T$ when $T_r^i \leq T$. 

For $i \in I$, arguing as in the proof of Proposition \ref{prop:bh_general} shows
\begin{equation} \nonumber
\sum_{r=1}^m \frac{1}{r} \P \left( T_r^i \leq T \ , \ T_r \ \leq \ T < T_r^{+} \ , \ T < \infty \right) \leq \frac{\alpha}{m} \sum_{k=1}^m \frac{1}{k}. 
\end{equation}
The proof is completed on application of (\ref{eq:seq_fdr_requirement}) to the terms in the first expansion with $i \notin I$ and re-ordering of the resulting terms.
\hfill$\Box$
\end{proof}

BH-I does not commute with always validity:

\begin{example}
Let $m=4$, with $p_{i,n}$, $i \leq 3$, be a.s. constant across $n$ with $p_{1,1} \sim U(0,1)$, and let $p_{4,1} = 1, p_{2,n} = 0$ for $n \geq 2$. These are feasible always valid p-value processes when the 1st  3 hypotheses are null and the 4th is non-null.  Consider the following stopping time: $T = 1$ if any null hypothesis is less than $\alpha/4$, any two nulls are less than $\alpha /2$, or all three nulls are less than $3 \alpha / 4$, and otherwise $T=2$. Straightforward calculation shows the standard BH procedure applied to $\v{p}_T$ gives an FDR of $\alpha + \frac{\alpha}{16} ( 2 - 9 \alpha + 45 \alpha^2 ) > \alpha$ for all $0 < \alpha < 1$.
\end{example}

%
%

\subsection{q-values}

For Bonferroni, the q-values are given by
$$ q^i = (p^i m) \wedge 1. $$
For BH with independence or general dependence respectively,
$$ q^{(j)} = \min_{k \geq j} \left ( \frac{p^{(k)} m} {k} \right ) \wedge 1 \,\, \text{or} \,\, \min_{k \geq j} \left ( \frac{p^{(k)} m \sum_{r = 1}^m 1/r}{k} \right ) \wedge 1. $$
The q-values for both Bonferroni and BH-I are currently displayed on the industry platform in different contexts.

\subsection{FCR control}

\begin{proof}{Proof of Theorem \ref{thm:fixed_horizon_FCR}.}
By Lemma 1 in \cite{benjamini2005false},
$$ FCR = \sum_{i = 1}^m \sum_{r = 1}^m \frac{1}{r} \P(|J \cup S^{BH}| = r, i \in J \cup S^{BH}, \theta^i \notin \tilde{CI}^i) $$
On the event $i \in J \cup S^{BH}$, there are two possibilities.  If $i \in S^{BH}$, we can say $R^{BH} \leq |J \cup S_{BH}|$.  If $i \notin S^{BH}$, we can say further that $R^{BH} + 1 \leq |J \cup S_{BH}|$.  In either case, it follows that $CI^i(1 - \alpha |J \cup S_{BH}|/m) \subset \tilde{CI}^i$, and so the FCR is at most
$$ \sum_{i = 1}^m \sum_{r = 1}^m \frac{1}{r} \P(|J \cup S^{BH}| = r, i \in J \cup S^{BH}, \theta^i \notin CI^i(1-\alpha r / m)) $$
\begin{enumerate}
\item[] {\em Case 1}: 
$i \notin J$.
\begin{align*}
\{& |J \cup S^{BH}| = r, i \in J \cup S^{BH}, \theta^i \notin CI^i(1-\alpha r / m) \}\\
&= \{ |J \cup (S^{BH})_0^{-i}| = r-1, p^i \leq \alpha r / m, \theta^i \notin CI^i(1- \alpha r / m) \}\\
&\subset \{ |J \cup (S^{BH})_0^{-i}| = r-1, \theta^i \notin CI^i(1- \alpha r / m) \}
\end{align*}
These two events are independent, so
\begin{align*}
&\sum_{r = 1}^m \frac{1}{r} \P(|J \cup S^{BH}| = r, i \in J \cup S^{BH}, \theta^i \notin CI^i(1-\alpha r / m))\\
&\leq \sum_{r = 1}^m \frac{1}{r} \P(|J \cup (S^{BH})_0^{-i}| = r-1) \P(\theta^i \notin CI^i(1- \alpha r / m))\\
&\leq \frac{\alpha}{m} \sum_{r = 1}^m \P(|J \cup (S^{BH})_0^{-i}| = r-1) = \frac{\alpha}{m}
\end{align*}
\item[] {\em Case 2}: 
$i \in J$.
\begin{align*}
&\sum_{r = 1}^m \frac{1}{r} \P(|J \cup S^{BH}| = r, i \in J \cup S^{BH}, \theta^i \notin CI^i(1-\alpha r / m))\\
&\leq \sum_{r = 1}^m \frac{1}{r} \P(|J \cup S^{BH}| = r, \theta^i \notin CI^i(1-\alpha r / m))\\
&= \sum_{r = 1}^m \frac{1}{r} \P(|J \cup S^{BH}| = r \, | \, \theta^i \notin CI^i(1-\alpha r / m)) \P(\theta^i \notin CI^i(1-\alpha r / m))\\
&\leq \frac{\alpha}{m} \sum_{r = 1}^m \P(|J \cup S^{BH}| = r \, | \, \theta^i \notin CI^i(1-\alpha r / m))
\end{align*}
Since $S^{BH}$ is a function only of the p-values and the data streams are independent, the events $\{ |J \cup S^{BH}| = r \}$ and $\{ \theta^i \notin CI^i(1-\alpha r / m)) \}$ are conditionally independent given $p^i$.  Hence,
$$ \P(|J \cup S^{BH}| = r \, | \, \theta^i \notin CI^i(1-\alpha r / m)) \leq \max_\rho \P(|J \cup S^{BH}| = r \, | \, p^i = \rho) $$
It is easily seen that this maximum must be attained at either $\rho = 0$ or $\rho = 1$, so
\begin{align*}
\P(|J \cup S^{BH}| = r \, | \, \theta^i &\notin CI^i(1-\alpha r / m))\\
&\leq \P(|J \cup S^{BH}| = r \, | \, p^i = 0) + \P(|J \cup S^{BH}| = r \, | \, p^i = 1)\\
&= \P(|J \cup (S^{BH})_0^{-i} \backslash i | = r-1) + \P(|J \cup (S^{BH})_1^{-i} \backslash i | = r-1)
\end{align*}
Thus
\begin{align*}
&\sum_{r = 1}^m \frac{1}{r} \P(|J \cup S^{BH}| = r, i \in J \cup S^{BH}, \theta^i \notin CI^i(1-\alpha r / m))\\
&\leq \frac{\alpha}{m} \left \{ \sum_{r = 1}^m \P(|J \cup (S^{BH})_0^{-i} \backslash i | = r-1) + \sum_{r = 1}^m \P(|J \cup (S^{BH})_1^{-i} \backslash i | = r-1) \right \}\\
&= \frac{2\alpha}{m}
\end{align*}
\end{enumerate}
Summing over all $i$ now gives the desired result.
\hfill$\Box$
\end{proof}

\begin{proof}{Proof of Theorem \ref{thm:sequential_FCR}.}
By the same argument as in Theorem \ref{thm:fixed_horizon_FCR}, we find that the FCR is at most
$$ \sum_{i = 1}^m \sum_{r = 1}^m \frac{1}{r} \P(|J \cup S^{BH}_T| = r, i \in J \cup S^{BH}_T, \theta^i \notin CI^i_T(1-\alpha r / m), T < \infty) $$
\begin{enumerate}
\item[] {\em Case 1}:
$i \notin J$. As in Theorem \ref{thm:fixed_horizon_FCR}, we obtain
\begin{align*}
&\sum_{r = 1}^m \frac{1}{r} \P(|J \cup S^{BH}_T| = r, i \in J \cup S^{BH}_T, \theta^i \notin CI^i_T(1-\alpha r / m), T < \infty)\\
&\leq \sum_{r = 1}^m \frac{1}{r} \P(|J \cup (S^{BH}_T)_0^{-i}| = r-1, \theta^i \notin CI^i_T(1- \alpha r / m), T < \infty)\\
&= \sum_{r = 1}^m \frac{1}{r} \P((T_{r-1})^{-i, J}_0 \leq T < (T_{r-1})^{-i, J+}_0, T_r^{i, \theta^i} \leq T < \infty)\\
&\leq \frac{\alpha}{m} \sum_{r = 1}^m \P((T_{r-1})^{-i, J}_0 \leq T < (T_{r-1})^{-i, J+}_0 | T_r^{i, \theta^i} \leq T < \infty)\\
&\leq \frac{\alpha}{m}
\end{align*}
\item[] {\em Case 2}:
$i \in J$.  As before,
\begin{align*}
\sum_{r = 1}^m \frac{1}{r} &\P(|J \cup S_T^{BH}| = r, i \in J \cup S_T^{BH}, \theta^i \notin CI_T^i(1-\alpha r / m))\\
&\leq \frac{\alpha}{m} \sum_{r = 1}^m \P(|J \cup S_T^{BH}| = r \, | \, \theta^i \notin CI_T^i(1-\alpha r / m))\\
&= \frac{\alpha}{m} \sum_{r = 1}^m \P(|J \cup S_T^{BH}| = r \, | \, T_r^{i, \theta^i} \leq T < \infty)\\
&\leq \frac{\alpha}{m} \sum_{r = 1}^m \max_{\rho} \P(|J \cup S_T^{BH}| = r \, | \, p_T^i = \rho, T_r^{i, \theta^i} \leq T < \infty)\\
&\leq \frac{\alpha}{m} \left \{ \sum_{r = 1}^m \P(|J \cup (S^{BH}_T)_0^{-i} \backslash i | = r-1 \, | \, T_r^{i, \theta^i} \leq T < \infty) + \sum_{r = 1}^m \P(|J \cup (S^{BH}_T)_1^{-i} \backslash i | = r-1 \, | \, T_r^{i, \theta^i} \leq T < \infty) \right \}\\
&= \frac{\alpha}{m} \Big \{ \sum_{r = 1}^m \P((T_{r-1})^{-i, J}_0 \leq T < (T_{r-1})^{-i, J+}_0 \, | \, T_r^{i, \theta^i} \leq T < \infty)\\
&\,\,\,\,\,\,\,\,\,\,\,\,\,\,\,\,\,\,\,\,\,\,\,\,\,\,\,\,\,\,\,\,\,\,\,\,\, + \sum_{r = 1}^m \P((T_{r-1})^{-i, J}_1 \leq T < (T_{r-1})^{-i, J+}_1 \, | \, T_r^{i, \theta^i} \leq T < \infty) \Big \}\\
&\leq \frac{2\alpha}{m}
\end{align*}
\end{enumerate}
Finally we sum over $i$.
\hfill$\Box$
\end{proof}

\end{APPENDICES}

\end{document}